\documentclass[10pt]{amsart}

\usepackage{amsmath,xspace,amssymb,mathrsfs}
\usepackage{color}

\input xy
\xyoption{all}
\xyoption{2cell}
\UseAllTwocells
\CompileMatrices

\newcommand{\Spec}{\operatorname{Spec}}
\renewcommand{\phi}{\varphi}
\newcommand{\rn}{\operatorname{rank}}
\newcommand{\Ker}{\operatorname{Ker}}

\newcommand{\Ima}{\operatorname{Im}}

\newcommand{\Max}{\operatorname{Max}}

\newcommand{\Min}{\operatorname{Min}}
\newcommand{\Ann}{\operatorname{Ann}}
\newcommand{\Hom}{\operatorname{Hom}}

\newcommand{\Cl}{\operatorname{\mathfrak{C}}}

\newcommand{\Pic}{\operatorname{Pic}}

\newcommand{\End}{\operatorname{End}}

\newtheorem{proposition}{Proposition}[section]
\newtheorem{lemma}[proposition]{Lemma}

\newtheorem{corollary}[proposition]{Corollary}
\newtheorem{theorem}[proposition]{Theorem}

\theoremstyle{definition}
\newtheorem{definition}[proposition]{Definition}
\newtheorem{example}[proposition]{Example}

\newtheorem{remark}[proposition]{Remark}

\usepackage{etoolbox}
\makeatletter
\patchcmd{\@settitle}{\uppercasenonmath\@title}{}{}{}
\patchcmd{\@setauthors}{\MakeUppercase}{}{}{}
\makeatother

\begin{document}

\title[Ideal class group and Picard group]{Ideal class group of an extension of rings and Picard group}

\author[A. Tarizadeh]{Abolfazl Tarizadeh}

\address{Department of Mathematics, Faculty of Basic Sciences, University of Maragheh, Maragheh, East Azerbaijan Province, Iran.}
\email{ebulfez1978@gmail.com}

\date{}
\subjclass[2020]{13B02, 14C22, 13C10, 11R29, 14C20}
\keywords{Invertible module; Ideal class group; Picard group}

\begin{abstract} For any extension of commutative rings $A\subseteq B$, by using invertible ideals, we first define an Abelian group $\Cl(A,B)$, that we call the ideal class group of this extension. Then we study the main properties of this group. Among them, we prove that the group $\Cl(A,B)$ is indeed the kernel of the natural group morphism $\Pic(A)\rightarrow \Pic(B)$ which is given by $L\mapsto L\otimes_{A}B$.
Then we show that both the classical ideal class group and, surprisingly, the Picard group are special cases of this structure. Next, we prove that there is a canonical isomorphism of groups $\Cl(A,B)\simeq\Cl(A_{\mathrm{red}},
B_{\mathrm{red}})$. We also show that if $A$ has finitely many maximal ideals, then the group $\Cl(A,B)$ is trivial. 
In the next key result, for any ring $A$, we obtain a canonical group map $\mathfrak{C}(A)\rightarrow
\mathfrak{C}(A_{\mathrm{red}})$ which is always injective where $\mathfrak{C}(A):=\Cl(A,T(A))$ and $T(A)$ is the total ring of fractions of $A$. In the same vein, we find an important example showing that this canonical injection is not necessarily surjective, and hence in general the ideal class group does not behave like the Picard group.\\
Next, using ideas from algebraic geometry, we prove that every commutative ring has a (faithfully flat) ring extension whose Picard group is trivial.
As an application, we show that every invertible module has the avoidance property.  \\
The Picard group does not necessarily commute with the ``infinite" direct product of rings. However, we show that the Picard group of the direct product of an (infinite) family of rings can be canonically embedded in the direct product of the Picard groups of the corresponding factor rings. As an application, the Picard group of an (infinite) direct product of rings is trivial if and only if the Picard group of each factor ring is trivial. We also prove that the Picard group of every zero-dimensional ring is trivial. 
\end{abstract}

\maketitle

\section{Introduction}

The subject of this article lies at the  intersection of commutative algebra,  algebraic geometry, and algebraic number theory. 

The classical ideal class group of an integral domain is of special importance in algebraic number theory and understanding the structure of this group is indeed one of the central goals of this field. 

The first main aim of this article is to generalize naturally this group to the more general setting of an arbitrary extension of commutative rings. In fact, for any extension of commutative rings $A\subseteq B$, just as in the classical construction, we define $\Cl(A,B)$ as the group of invertible ideals of this extension modulo the subgroup of principal invertible ideals. We call $\Cl(A,B)$ the ideal class group of the extension $A\subseteq B$.  

We will then examine the relationship of this group to the Picard groups of the corresponding rings in the extension. In fact, after overcoming some technical difficulties arising from this generalization, then in Theorem \ref{Thm TRS 2}, we obtain the following exact sequence of Abelian groups: $$\xymatrix{0\ar[r]&\mathfrak{C}(A,B)\ar[r]&\Pic(A)
\ar[r]&\Pic(B).}$$
This result indeed tells us that
the ideal class group $\Cl(A,B)$ is naturally isomorphic to the kernel of the canonical group morphism $\Pic(A)\rightarrow\Pic(B)$ that is given by $L\mapsto L\otimes_{A}B$. To prove this theorem, one of the key results that applied is Lemma \ref{key lemma} which characterizes invertible ideals in terms of invertible modules. It is interested in noting that the converse is also true, i.e., every invertible module can be realized as an invertible ideal of a certain ring extension (see the proof of Corollary \ref{coro toofan 21}).

Next, in Theorem \ref{Lemma TD 6}, we obtained an unexpected result which states that every commutative ring has a faithfully flat ring extension whose Picard group is trivial. 

This result had an exciting story for us. Indeed, at first we formulated this theorem merely as a conjecture. Next, we asked Pierre Deligne whether this conjecture was true or not. Then, he answered that the conjecture was correct and sent us a geometrical proof of it. 
However, his very sketchy proof was too difficult for us to understand and was based entirely on modern algebraic geometry. In fact, his proof was based on the following two main ideas. The first idea was that line bundles (projective modules everywhere of rank one) are classified by $\mathbb{G}_{\mathrm{m}}$-torsors (principal homogeneous spaces). The second idea was that the universal way to vanish (trivialize) a torsor under a group scheme is to base change by this same torsor. 

Then, using the algebraic interpretations of these geometric ideas, we (with the collaboration of Melvin Hochster) were able to transform Deligne's proof into a purely algebraic proof. In fact, from the algebraic point of view, one of the key ideas of the proof is that the push-out of every line bundle $L\in\Pic(A)$ is vanished (trivialized) by its own torsor algebra $\bigoplus\limits_{n\in\mathbb{Z}}L^{\otimes n}$ which is a $\mathbb{Z}$-graded commutative ring. However, instead of this torsor algebra, we found a much simpler algebra for vanishing the line bundle $L$ (see the proof of theorem).
The second main idea is that, by taking the tensor products of the algebras vanishing each of the line bundles, one can vanish all them together (then repeat this process to vanish the new line bundles which could have appeared). The final idea of the proof is that the Picard group commutes with direct limits. 

Next, in Theorem \ref{key lemma 2 Pic}, we prove that the Picard group of the direct product of an (infinite) family of rings can be canonically embedded in the direct product of the Picard groups of the factor rings. In particular, the Picard group of the direct product of an arbitrary family of rings is trivial if and only if the Picard group of each factor ring is trivial. As an application, a weak version of the above theorem is easily deduced which states that every commutative ring has a ring extension whose Picard group is trivial. 

In Theorem \ref{lemma toofan 20}, we prove another important result which asserts that for a given extension of rings $A\subseteq B$, if $A$ has finitely many maximal ideals then every invertible ideal of this extension is principal.

Then some applications of the above theorems are given. 
We show that in a given extension of rings $A\subseteq B$, if $B$ has finitely many maximal ideals then we have a canonical isomorphism of groups $\Cl(A,B)\simeq\Pic(A)$. As another application, for any commutative ring $A$, we obtain the following exact sequence of Abelian groups: $$\xymatrix{0\ar[r]&\mathfrak{C}(A)\ar[r]&\Pic(A)\ar[r]&
\Pic(T(A))}$$ where $\Cl(A):=\Cl(A, T(A))$ and $T(A)$ is the total ring of fractions of $A$. 
In particular, if $A$ is a reduced ring with finitely many minimal primes (e.g. an integral domain), then we have a canonical isomorphism of groups $\Cl(A)\simeq\Pic(A)$. 

These applications culminate in showing that the Picard group is a special case of the ideal class group. More precisely, we show that the Picard group of every commutative ring is canonically isomorphic to the ideal class group of a certain extension of that ring (see Corollary \ref{Coro culmination}). 
These applications indeed show that this new concept of the ideal class group of an extension unifies classical structures.

As another important application, in Corollary \ref{Coro JPAA}, we prove that every invertible module (line bundle) has the avoidance property. This result considerably generalizes \cite[Theorem 1.5]{Quartararo-Butts} and our recent result \cite[Corollary 3.2]{Tarizadeh-Chen}. 

In Theorem \ref{Thm zero dim}, we prove another interesting result which asserts that the Picard group of every zero-dimensional ring is trivial. 

We learned that the classical ideal class group is also denoted by some authors as ``Cl''. However, this indication has already been used in the literature to denote the Weil divisor class group. Therefore, to avoid confusion, we will not use this notation to denote the ideal class group. Also, a reviewer has called our attention to the article \cite{Roberts-Singh}. We then observed that some exact sequences similar to the exact sequence we obtained in Theorem \ref{Thm TRS 2} can be found in the literature \cite[Chap IX, \S3, Theorem 3.3]{Bass} and \cite[Theorem 2.4]{Roberts-Singh}. Although there are fewer connections between the exact sequence obtained in \cite[Chap IX, \S3, Theorem 3.3]{Bass} and Theorem \ref{Thm TRS 2} (on this subject, see also Remark \ref{Remark bir}), our result has very close connections with \cite[Theorem 2.4]{Roberts-Singh}.

Next, we will study the behaviour of the ideal class group for a tower of extensions of rings. In fact, in Theorem \ref{Theorem fourth TS}, for any extensions of commutative rings $A\subseteq B \subseteq C$, we obtain the following exact sequence of Abelian groups:  $$\xymatrix{0\ar[r]&
\mathfrak{C}(A,B)\ar[r]&\mathfrak{C}(A,C)\ar[r]&
\mathfrak{C}(B,C).}$$
We also prove that under a gentle condition the above sequence is also right exact (see Corollary \ref{Coro Ani-Elman-Asena}). Similarly, an exact sequence about the groups of invertible ideals: $$\xymatrix{0\ar[r]&
\mathscr{G}(A,B)\ar[r]&\mathscr{G}(A,C)\ar[r]&
\mathscr{G}(B,C)}$$ is obtained first in \cite[\S3]{Singh} and later in \cite[Lemma 4.3]{Sahhu-Weibel}. However, as we will see in the proof of Theorem \ref{Theorem fourth TS}, it is quite easy to prove the exactness of this sequence. In fact, Theorem \ref{Theorem fourth TS} improves this result significantly. 
  
As another main result, in Theorem \ref{theorem TH}, we show that for any extension of commutative rings $A\subseteq B$, we have a canonical isomorphism of groups $\Cl(A,B)\simeq\Cl(A_{\mathrm{red}},
B_{\mathrm{red}})$. Note that a result similar to Theorem \ref{theorem TH} was proved in \cite[Theorem 4.1]{Sahhu-Weibel}, but the group $L\mathcal{I}(A,B)$ defined in that article is a completely different concept and has no connection with our ideal class group $\Cl(A,B)$. For instance, $L\mathcal{I}(A,B)$ is always a torsion-free Abelian group (see \cite[Corollary 3.6]{Sahhu-Weibel}), but for some certain extensions $A\subseteq B$ the group $\Cl(A,B)$ will be isomorphic to the additive torsion group $\mathbb{Z}/d$ for some $d\geqslant2$.  

We also show that for any commutative ring $A$ we obtain a canonical group morphism $\mathfrak{C}(A)\rightarrow
\mathfrak{C}(A_{\mathrm{red}})$ that is always injective (see Corollary \ref{Coro ilginch}). In the same vein, we find a key example showing that, unlike the Picard group, this canonical injection is not an isomorphism in general (see Example \ref{Example TDH}). In this example we heavily use the whole strength of the Nagata idealization tool to achieve the goal.  

Although the total ring of fractions in general is not a functorial construction, for some extensions of rings $A\subseteq B$ that makes it functorial (including all extensions of integral domains), we obtain the following exact sequence of Abelian groups (see Theorem \ref{Coro ilginch uch}): $$\xymatrix{0\ar[r]&
\mathfrak{C}(A,B)\ar[r]&\mathfrak{C}(A)\ar[r]&
\mathfrak{C}(B).}$$

Next we will examine the vanishing aspect of the ideal class group. In this regard, in Corollary \ref{Lemma TH 2}, we show that if an extension of rings has a retraction, then its ideal class group is trivial. As an application, we show that for any commutative ring $A$ and for any commutative monoid $M$, the ideal class group of the monoid-ring extension $A\subseteq A[M]$ is trivial. In particular, the ideal class groups of the polynomial ring extensions $A\subseteq A[x]$ and $A\subseteq A[x,x^{-1}]$ are trivial. The ideal class group of the extension $A\subseteq A[[x]]$ is also trivial. In this regard, we also show that the ideal class group of any extension of rings $A\subseteq B$ with $A$ has finitely many maximal ideals is trivial as well (see Theorem \ref{lemma toofan 20}). 
In other words, every invertible ideal of these extensions is principal. 
 
Finally, we give another important example showing that there are some ring extensions $A\subseteq B$ such that the kernel of the canonical group map $\Pic(A)\rightarrow\Pic(B)$ is very sensitive to changes in scalars (see Example \ref{Example 3 three}). 

\section{Preliminaries}

In this section, we recall some basic background for the convenience of the reader. The group of units (invertible elements) of a commutative ring $A$ is denoted by $A^{\ast}$. The set of zero-divisors of $A$ is denoted by $Z(A)=\{a\in A: \exists 0\neq b \in A, ab=0\}$. By $A_{\mathrm{red}}$ we mean the ring $A$ modulo its nil-radical. If $r$ is an element of a ring $A$, then $D(r)=\{\mathfrak{p}\in\Spec(A): r\notin\mathfrak{p}\}$. 

Recall that if $F=\bigoplus\limits_{k\in S}A$ is a free module over a ring $A$ and $M$ is an $A$-module, then we have the canonical isomorphism of $A$-modules $F\otimes_{A}M\rightarrow\bigoplus\limits_{k\in S}M$ which sends each pure tensor $(r_{k})\otimes x$ into $(r_{k}x)$. The inverse map $\bigoplus\limits_{k\in S}M\rightarrow F\otimes_{A}M$ is given by $(x_{k})\mapsto\sum\limits_{k\in S}\epsilon_{k}\otimes x_{k}$ where $\epsilon_{k}=(\delta_{i,k})_{i\in S}$ and $\delta_{i,k}$ is the Kronecker delta.

If $M$ is a finitely generated flat module over a ring $A$, then for each $\mathfrak{p}\in\Spec(A)$, there exists a (unique) natural number $n_{\mathfrak{p}}\geqslant0$ such that $M_{\mathfrak{p}}\simeq
(A_{\mathfrak{p}})^{n_{\mathfrak{p}}}$ as  $A_{\mathfrak{p}}$-modules, because it is well known that every finitely generated flat module over a local ring is a free module (see e.g. \cite[Theorem 7.10]{Matsumura}). 
Hence, we obtain a function  $\mathrm{r}_{M}:\Spec(A)\rightarrow\mathbb{Z}$ given by $\mathfrak{p}\mapsto n_{\mathfrak{p}}$. We call $\mathrm{r}_{M}$ the rank map of $M$. 

It can be seen that this number $\rn_{R_{\mathfrak{p}}}(M_{\mathfrak{p}}):=n_{\mathfrak{p}}$ is the dimension of $\kappa(\mathfrak{p})$-vector space $M\otimes_{A}\kappa(\mathfrak{p})$ where $\kappa(\mathfrak{p})=A_{\mathfrak{p}}/
\mathfrak{p}A_{\mathfrak{p}}$ is the residue field of $A$ at $\mathfrak{p}$.

If $M$ and $N$ are finitely
generated flat (resp. projective) $R$-modules, then $M\oplus N$ and $M\otimes_{R}N$ are finitely generated flat (resp. projective) $R$-modules and we have $\mathrm{r}_{M\oplus N} = \mathrm{r}_{M}+\mathrm{r}_{N}$ and $\mathrm{r}_{M\otimes_{R}N} = \mathrm{r}_{M}\cdot\mathrm{r}_{N}$.

It is well known that the rank map of a finitely generated flat $A$-module is continuous if and only if it is a projective $A$-module. 

By invertible module over a ring $A$ we mean a finitely generated projective $A$-module of constant rank 1. 

Let $M$ and $N$ be modules over a ring $A$. If $M\otimes_{A}N\simeq A^{n}$ as $A$-modules for some $n\geqslant1$, then $M$ and $N$ are finitely generated projective $A$-modules. In particular, if $n=1$, then $M$ and $N$ are invertible $A$-modules. 

It can be seen that a module $M$ over a ring $A$ is invertible (i.e., finitely generated projective $A$-module of constant rank 1) if and only if there are (finitely many) elements $f_{1},\ldots,f_{n}\in A$ which generate the unit ideal of $A$ and $M_{f}$ is a free $A_{f}$-module of constant rank 1 for all $f\in\{f_{1},\ldots,f_{n}\}$,
or equivalently, the canonical morphism of $A$-modules $\widehat{M}\otimes_{A}M\rightarrow A$ given by $f\otimes x\mapsto f(x)$ is an isomorphism where the $A$-module $\widehat{M}=\Hom_{A}(M,A)$ is called the dual of $M$. In this case, $\widehat{M}$ is also an invertible $A$-module. 

All these results are standard and can be found in the literature (see e.g. \cite{Bass}, \cite{Silvester}, \cite{Swan} and \cite{Weibel}).

For any commutative ring $A$, by $\Pic(A)$ we mean the set of isomorphism classes of invertible modules over $A$ (invertible modules are also called line bundles). 
This set $\Pic(A)$ by the  operation $[L_{1}]\cdot [L_{2}]=[L_{1}\otimes_{A}L_{2}]$ is an Abelian group whose identity element is the isomorphism class of $A$, and the inverse of each element $[L]\in\Pic(A)$ is the isomorphism class of $\widehat{L}$. The group $\Pic(A)$ was introduced by Grothendieck and given the name \emph{Picard group of $A$}. We will often denote the isomorphism class $[L]$ simply by $L$ if there is no confusion.
 
If $f:A\rightarrow B$ is a morphism of rings, then the map $\Pic(f):\Pic(A)\rightarrow\Pic(B)$ given by $L\mapsto L\otimes_{A}B$ is a morphism of groups. This group morphism in general is not necessarily injective, even if $f$ is injective. 

The Picard group construction is indeed a covariant functor from the category of commutative rings to the category of Abelian groups. 

It can be seen that the Picard group commutes with direct limits. More precisely, if $(A_{k}, f_{i,k})$ is a direct system of rings over a (directed) poset $I$, then the groups $\Pic(A_{k})$ together with the maps $\Pic(f_{i,k})$, as transition morphisms, is a direct system of Abelian groups over the same (directed) poset and we have the canonical isomorphism of groups $\Pic(\lim\limits_{\overrightarrow{k\in I}}A_k)\simeq\lim\limits_{\overrightarrow{k\in I}}\Pic(A_k)$. 

The Picard group also commutes with ``finite'' direct products of rings. In particular, for given rings $A$ and $B$ consider their direct product ring $R=A\times B$, then the morphism of groups $\Pic(R)\rightarrow\Pic(A)\times\Pic(B)$ given by $L\mapsto(L\otimes_{R}A, L\otimes_{R}B)$ is an isomorphism. In this regard see also Theorem \ref{key lemma 2 Pic}.

The following well known theorem is another fundamental result on Picard groups:

\begin{theorem}\label{Thm 100 npf} Let $I$ be an ideal of a ring $A$ and consider the canonical ring map $f:A\rightarrow A/I$. If $I$ is contained in the Jacobson radical of $A$, then the map $\Pic(f):\Pic(A)\rightarrow\Pic(A/I)$ is injective. If moreover, $I$ is contained in the nil-radical of $A$, then the map $\Pic(f)$ is an isomorphism. In particular, we have a canonical isomorphism of groups $\Pic(A)\simeq\Pic(A_{\mathrm{red}})$.   
\end{theorem}

Parts of the proof can be found in \cite[Chap III, Proposition 2.12]{Bass} and \cite[Chap X, \S5, Theorem 5.10, 1.(b)]{Lombardi}), but finding an integrated and single proof of this theorem, especially the surjectiveness of the map $\Pic(f)$, is a rather difficult task. Hence, we provide a proof for it. 

First recall that if $A\rightarrow B$ is a surjective morphism of rings then for any $B$-modules $M$ and $N$ we have $\Hom_{B}(M,N)=\Hom_{A}(M,N)$.

It is well known that if $M$ and $N$ are finitely generated projective $A$-modules such that $M/IM\simeq N/IN$ and $I$ is contained in the Jacobson radical of $A$, then $M\simeq N$. This establishes the injectivity of $\Pic(f)$. 

To see surjectivity, let $L$ be an invertible module over $A/I$. There is a surjective morphism of $A$-modules $\phi:(A/I)^{n}\rightarrow L$ for some $n\geqslant1$, and since $L$ is a projective module there is a morphism of $A$-modules $\psi:L\rightarrow (A/I)^{n}$ such that $\phi\circ\psi$ is the identity map. Then $g=\psi\circ\phi:(A/I)^{n}\rightarrow (A/I)^{n}$ is a morphism of $A$-modules such that $L\simeq\Ima(g)$ and $g$ is an idempotent of the ring $\End((A/I)^{n})$, i.e., $g\circ g=g$. Then we show that the natural ring map $\End(f):\End(A^{n})\rightarrow\End((A/I)^{n})$ lifts idempotents. 

Recall that if $M$ is an $A$-module, then any ring map $A\rightarrow B$ induces the ring map $\End_{A}(M)\rightarrow\End_{B}(M\otimes_{A}B)$ that is given by $h\mapsto h\otimes1_{B}$. 

We denote by $\mathscr{M}_{n}(A)$ the ring of $n\times n$ square matrices over $A$. There is a natural isomorphism of rings $\mathscr{M}_{n}(A)\rightarrow\End(A^{n})$ which is given by $M\mapsto T_{M}$ where $T_{M}:A^{n}\rightarrow A^{n}$ is defined as $x\mapsto Mx$. The following diagram of rings is also commutative: $$\xymatrix{\mathscr{M}_{n}(A)
\ar[r]^{\mathscr{M}_{n}(f)\:\:\:\:\:}\ar[d]^{\simeq}&
\mathscr{M}_{n}(A/I)\ar[d]^{\simeq}
\\\End(A^{n})\ar[r]^{\End(f)\:\:\:\:\:\:\:}&
\End((A/I)^{n}).}$$ 
 
Hence, to prove that $\End(f)$ lifts idempotents, it suffices to show that $\mathscr{M}_{n}(f)$ lifts idempotents.  
 
The ring map $\mathscr{M}_{n}(f):
\mathscr{M}_{n}(A)\rightarrow\mathscr{M}_{n}(A/I)$ is surjective and $\Ker(\mathscr{M}_{n}(f))=\mathscr{M}_{n}(I)$, the set of all $n\times n$ matrices with entries in $I$.
We claim that the two-sided ideal $\mathscr{M}_{n}(I)$ is contained in the set of nilpotent elements of $\mathscr{M}_{n}(A)$. Indeed, take a matrix $M$ in $\mathscr{M}_{n}(I)$, and let $J$ be the ideal of $A$ generated by the entries of $M$. Then $A\in\mathscr{M}_{n}(J)$. Clearly $J$ is a nilpotent ideal, i.e., $J^{d}=0$  for some $d\geqslant1$ (because $J$ is a finitely generated ideal contained in the nilradical of $A$). For any ideals $I_{1}$ and $I_{2}$ of $A$ we have $\mathscr{M}_{n}(I_{1})\mathscr{M}_{n}(I_{2})\subseteq
\mathscr{M}_{n}(I_{1}I_{2})$.
This shows that $A^{d}=0$, and establishes the claim. 

It is well known that if a two-sided ideal of a ring (not necessarily commutative) is contained in the set of nilpotent elements, then the idempotents of this ring can be lifted modulo this ideal.
   
This shows that the natural ring map $\mathscr{M}_{n}(f)$ and so $\End(f)$ lift idempotents. Thus there exists a morphism of $A$-modules $h:A^{n}\rightarrow A^{n}$ such that $h\circ h=h$ and $h\otimes1_{A/I}=g$ with considering the identification $(A/I)^{n}=A^{n}\otimes_{A}A/I$. Since $h$ is idempotent, we have a decomposition of $A$-modules $A^{n}\simeq\Ker(h)\oplus\Ima(h)$. This shows that
$M:=\Ima(h)$ is a finitely generated projective $A$-module. So we have the following split exact sequence of $A$-modules: $$\xymatrix{0\ar[r]&K\ar[r]&A^{n}\ar[r]^{h'}&M\ar[r]&0}$$
where $h'(x):=h(x)$ and $K=\Ker(h')$. Then tensoring with $-\otimes_{A}A/I$ leaves the sequence split exact: 
$$\xymatrix{0\ar[r]&K\otimes_{A}A/I
\ar[r]&A^{n}\otimes_{A}A/I\ar[r]^{h'\otimes1}
&M\otimes_{A}A/I
\ar[r]&0.}$$
This shows that $M/IM\simeq\Ima(h'\otimes1)$. 
Note that $h=i\circ h'$ where $i: M\rightarrow A^{n}$ is the inclusion map. We have  $h\otimes1=(i\otimes1)\circ(h'\otimes1)$. The map $i\otimes1$ is injective because, similar to the above, the split exact sequence: $$\xymatrix{0\ar[r]&M\ar[r]^{i}&A^{n}
\ar[r]&\Ima(h)\ar[r]&0}$$ gives the following split exact sequence: $$\xymatrix{0\ar[r]&M\otimes_{A}A/I\ar[r]^{i\otimes1}&
A^{n}\otimes_{A}A/I\ar[r]&\Ima(h)\otimes_{A}A/I\ar[r]&0.}$$
Thus $\Ima(h\otimes1)\simeq\Ima(h'\otimes1)$. Then 
$M/IM\simeq\Ima(h\otimes1)
\simeq\Ima(g)\simeq L$. Finally, we need to show that the rank map of $M$ is the constant function 1. If $\mathfrak{p}$ is a prime ideal of $A$, then $I\subseteq\mathfrak{p}$. Thus $M_{\mathfrak{p}}/I_{\mathfrak{p}}M_{\mathfrak{p}}\simeq L_{\mathfrak{p}/I}\simeq A_{\mathfrak{p}}/IA_{\mathfrak{p}}$ is a cyclic module over $A_{\mathfrak{p}}/IA_{\mathfrak{p}}$. So by the Nakayama lemma, we get that $M_{\mathfrak{p}}$ is a cyclic $A_{\mathfrak{p}}$-module. But $M_{\mathfrak{p}}$ is also a nonzero free $A_{\mathfrak{p}}$-module. This shows that $M_{\mathfrak{p}}\simeq A_{\mathfrak{p}}$. 

\section{Ideal class group of an extension vs Picard group}

In mathematics, especially in algebraic number theory, invertible ideals and the ideal class group are defined for integral domains, more precisely, for such an extension $A\subseteq F$ where $A$ is an integral domain and $F$ is its field of fractions. Later in \cite{Roberts-Singh} and \cite{Singh}, invertible ideals were initially studied for the more general setting of arbitrary extensions of rings. 

In this section, we further develop this theory to pave the way for deeper results in the subsequent sections. To achieve this goal, let $A\subseteq B$ be an extension of commutative rings. If $L$ and $L'$ are $A$-submodules of $B$ then their multiplication $LL'$, the set of all finite sums of the form $\sum\limits_{k=1}^{n}x_{k}y_{k}$ with $n\geqslant1$, $x_{k}\in L$ and $y_{k}\in L'$, is also an $A$-submodule of $B$. We call an $A$-submodule $L$ of $B$ an \emph{invertible ideal of the extension $A\subseteq B$} if there exists an $A$-submodule $L'$ of $B$ such that $LL'=A$. In this case, it can be easily seen that $L'=\{b\in B: Lb\subseteq A\}$. Therefore such $L'$ is unique and is denoted by $L^{-1}$.
 
The set of invertible ideals of a ring extension $A\subseteq B$ is an Abelian group under multiplication with the identity element $A$. 
Following \cite{Roberts-Singh}, we will denote this group by $\mathscr{G}(A,B)$. Then we have the following simple observations:

\begin{lemma}\label{Lemma pii 20} Let $A\subseteq B$ be an extension of rings and $b\in B$. Then $L=Ab$ is an invertible ideal of this extension if and only if $b\in B^{\ast}$. 
In this case, $L^{-1}=Ab^{-1}$.
\end{lemma}

\begin{proof} If $L=Ab$ is an invertible ideal then $LL'=A$ for some $A$-submodule $L'$ of $B$. Thus $1=\sum\limits_{k=1}^{n}x_{k}y_{k}$ with $x_{k}=a_{k}b\in L$ and $y_{k}\in L'$ for all $k$. Then $1=(\sum\limits_{k=1}^{n}a_{k}y_{k})b$. This shows that $b\in B^{\ast}$. 
\end{proof}

\begin{lemma} If $L_{1}$ and $L_{2}$ are invertible ideals of an extension of rings $A\subseteq B$ with   $L_{1}\subseteq L_{2}$, then $L^{-1}_{2}\subseteq L^{-1}_{1}$.
\end{lemma}

\begin{proof} We have $L^{-1}_{2}=L^{-1}_{2}A=L^{-1}_{2}L_{1}L^{-1}_{1}
\subseteq L^{-1}_{2}L_{2}L^{-1}_{1}=
AL^{-1}_{1}=L^{-1}_{1}$. 
\end{proof}
 
\begin{definition} Let $A\subseteq B$ be an extension of rings. The group of invertible ideals $\mathscr{G}(A,B)$ modulo $H=\{Ax: x\in B^{\ast}\}$, the subgroup of principal invertible ideals, is called \emph{the ideal class group of this extension} and is denoted by $\mathfrak{C}(A,B)$.
\end{definition}

For any extension of rings $A\subseteq B$ we have the following short exact sequence of Abelian groups: $$\xymatrix{0\ar[r]& B^{\ast}/A^{\ast}\ar[r]&\mathscr{G}(A,B)\ar[r]&
\mathfrak{C}(A,B)\ar[r]&0}$$
where the group morphism $B^{\ast}/A^{\ast}\rightarrow\mathscr{G}(A,B)$ is defined as $xA^{\ast}\mapsto Ax$.
 
To study these groups systematically, we need some simple categorical concepts. By the category of extensions of rings we mean a category
whose objects are the pairs $(A,B)$ where $A\subseteq B$ is an extension of rings, and whose morphisms are the arrows $\phi:(A,B)\rightarrow(A',B')$ where $\phi:B\rightarrow B'$ is a morphism of rings such that $\phi(A)\subseteq A'$. We then show that the group of invertible ideals and the ideal class group constructions are indeed covariant functors from this category to the category of Abelian group. 

To this end, let $\phi:(A,B)\rightarrow(A',B')$ be a morphism in the category of extensions of rings. If $L$ is an $A$-submodule of $B$ then $L_{\phi}:=\sum\limits_{x\in L}A'\phi(x)$ is an $A'$-submodule of $B'$. It is clear that $A_{\phi}=A'$ and $(LL')_{\phi}=L_{\phi}L'_{\phi}$ for any $A$-submodules $L$ and $L'$ of $B$. Thus if $L$ is an invertible ideal of the extension $A\subseteq B$ then $L_{\phi}$ is an invertible ideal of the extension $A'\subseteq B'$. In fact, the map $\mathscr{G}(A,B)\rightarrow\mathscr{G}(A',B')$ given by $L\mapsto L_{\phi}$ is a morphism of groups. If $x\in B^{\ast}$ then $\phi(x)\in (B')^{\ast}$ and $(Ax)_{\phi}=A'\phi(x)$.
Hence, $\phi$ induces the desired group morphism $\Cl(A,B)\rightarrow\Cl(A',B')$ which is given by $LH\mapsto L_{\phi}H'$ where $H=\{Ax: x\in B^{\ast}\}$ and $H'=\{A'x: x\in (B')^{\ast}\}$. 

If $L$ and $L'$ are $A$-submodules of $B'$, then their multiplication $LL'$ is defined similarly above which is an $A$-submodule of $B'$. In particular, if $L$ is an $A$-submodule of $B$ then $L_{\phi}=\phi(L)A'$. If $L$ is an $A$-submodule of $B$ and $S$ is a multiplicative subset of $A$, then $L_{\phi}=S^{-1}L$ where $\phi:B\rightarrow S^{-1}B$ is the canonical ring map.
In particular, if $L$ and $L'$ are $A$-submodules of $B$, then $S^{-1}(LL')=(S^{-1}L)(S^{-1}L')$. 

To prove the first main result of this article, we need the following key lemmas: 

\begin{lemma}\label{Lemma 5 besh} Let $\phi:(A,B)\rightarrow(A',B')$ be a morphism in the category of extensions of rings and $L$ be an invertible ideal of the extension $A\subseteq B$. Then for any $A$-submodule $L'$ of $B'$ we have the canonical isomorphism of $A$-modules $L\otimes_{A}L'\simeq\phi(L)L'=L_{\phi}L'$. In particular, $L\otimes_{A}A'\simeq L_{\phi}$ as $A'$-modules. 
\end{lemma}

\begin{proof} We show that the canonical morphism of $A$-modules $L\otimes_{A}L'\rightarrow\phi(L)L'$ given by  $b\otimes b'\mapsto\phi(b)b'$ is an isomorphism. This map is clearly surjective. To see injectivity, suppose that $\sum\limits_{i=1}^{n}\phi(b_i)b'_{i}=0$ where $b_{i}\in L$ and $b'_{i}\in L'$ for all $i$. Since $LL^{-1}=A$, we may write $1=\sum\limits_{k=1}^{d}x_{k}y_{k}$ with $x_{k}\in L$ and $y_{k}\in L^{-1}$ for all $k$. Then in $L\otimes_{A}L'$ we have $\sum\limits_{i=1}^{n}b_{i}\otimes b'_{i}=
\sum\limits_{i=1}^{n}
\big(b_{i}(\sum\limits_{k=1}^{d}x_{k}y_{k})\big)
\otimes b'_{i}=
\sum\limits_{i=1}^{n}(\sum\limits_{k=1}^{d}
x_{k}b_{i}y_{k})\otimes b'_{i}$. But each $b_{i}y_{k}\in LL^{-1}=A$. Then $\sum\limits_{i=1}^{n}b_{i}\otimes b'_{i}=
\sum\limits_{i=1}^{n}\big(\sum\limits_{k=1}^{d}
(x_{k}b_{i}y_{k}\otimes b'_{i})\big)
=\sum\limits_{i=1}^{n}\big(\sum\limits_{k=1}^{d}
x_{k}\otimes\phi(b_{i}y_{k})b'_{i}\big)=
\sum\limits_{i=1}^{n}\big(\sum\limits_{k=1}^{d}
x_{k}\otimes\phi(b_{i})\phi(y_{k})b'_{i}\big)=
\sum\limits_{k=1}^{d}
x_{k}\otimes\big(\phi(y_{k})\sum\limits_{i=1}^{n}
\phi(b_{i})b'_{i}\big)=0$. This completes the proof. 
\end{proof}

\begin{corollary}\label{Lemma TRRS} Let $A\subseteq B$ be an extension of rings and $L$ be an invertible ideal of this extension. Then for any $A$-submodule $L'$ of $B$ we have the canonical isomorphism of $A$-modules $L\otimes_{A}L'\simeq LL'$. 
\end{corollary}

\begin{proof} It follows from Lemma \ref{Lemma 5 besh}.
\end{proof}

\begin{lemma}\label{key lemma} Let $A\subseteq B$ be an extension of rings and $L$ be an $A$-submodule of $B$. Then the following assertions are equivalent: \\
$\mathbf{(i)}$ $L$ is an invertible ideal of the extension $A\subseteq B$. \\
$\mathbf{(ii)}$  $L$ is a finitely generated projective $A$-module of rank 1 and $LB=B$. 
\end{lemma}

\begin{proof} (i)$\Rightarrow$(ii): Since $LL^{-1}=A$,  we may write $1=\sum\limits_{k=1}^{n}x_{k}y_{k}$ where $x_{k}\in L$ and $y_{k}\in L^{-1}$ for all $k$. If $x\in L$ then $xy_{k}\in LL^{-1}=A$ for all $k$. Thus $x=\sum\limits_{i=1}^{n}(xy_{k})x_{k}
\in\sum\limits_{k=1}^{n}Ax_{k}$. Then $L=\sum\limits_{k=1}^{n}Ax_{k}$ is a finitely generated $A$-module. We also have $B=AB=L(L^{-1}B)\subseteq LB\subseteq B$ and so $LB=B$. Then we show that $L$ is a projective $A$-module. Consider the surjective morphism of $A$-modules $\phi:A^{n}\rightarrow L$ which sends each unit vector $\epsilon_{k}\in A^{n}$ into $x_{k}$. Then the map $\psi:L\rightarrow A^{n}$ given by $\psi(x)=(xy_{1},\ldots,xy_{n})$ is a morphism of $A$-modules and $\phi\circ\psi$ is the identity map of $L$. Thus the short exact sequence $\xymatrix{0\ar[r]&\Ker\phi\ar[r]&A^{n}\ar[r]^{\phi}
&L\ar[r]&0}$ splits and so $L$ is a projective $A$-module. Next, we show that $L$ is of constant rank 1. It suffices to show that for each $\mathfrak{p}\in\Spec(A)$ then $L_{\mathfrak{p}}\simeq A_{\mathfrak{p}}$ as $A_{\mathfrak{p}}$-modules. By Corollary \ref{Lemma TRRS}, we have $L\otimes_{A}L^{-1}\simeq A$ as $A$-modules. This yields that $L_{\mathfrak{p}}\otimes_{A_{\mathfrak{p}}}
(L^{-1})_{\mathfrak{p}}
\simeq A_{\mathfrak{p}}$. But $L_{\mathfrak{p}}\simeq(A_{\mathfrak{p}})^{m}$ and $(L^{-1})_{\mathfrak{p}}\simeq(A_{\mathfrak{p}})^{n}$ for some natural numbers $m,n\geqslant0$, because every finitely generated projective (even flat) module over a local ring is a free module. It follows that $mn=1$ and so $m=n=1$. This shows that $L$ is of rank 1.  \\
(ii)$\Rightarrow$(i): Since $L$ is a finitely generated $A$-module, we have a surjective morphism of $A$-modules $\phi:A^{n}\rightarrow L$ for some $n\geqslant1$. Since $L$ is $A$-projective, this map splits. So there exists a morphism of $A$-modules $\psi:L\rightarrow A^{n}$ such that $\phi\psi:L\rightarrow L$ is the identity map. But $L\otimes_{A}B$ is a finitely generated projective $B$-module of rank 1. So the canonical surjective morphism of $B$-modules $L\otimes_{A}B\rightarrow LB=B$ is an isomorphism. Indeed, it can be readily verified that every surjective morphism of modules between finitely generated projective modules with the same rank maps is an isomorphism. So we may consider the map $\psi\otimes1_{B}:B=L\otimes_{A}B\rightarrow A^{n}\otimes_{A}B=B^{n}$. Assume $(\psi\otimes1_{B})(1)=(b_{1},\ldots,b_{n})$ then $\sum\limits_{k=1}^{n}b_{k}x_{k}=1$  where $x_{k}=\phi(\epsilon_{k})\in L$ for all $k$. Let $L'=\sum\limits_{k=1}^{n}Ab_{k}$. Then $A\subseteq LL'$. It is clear that $\psi\otimes1_{B}$ is a morphism of $B$-modules. So for each $k$ we have $(x_{k}b_{1},\ldots,x_{k}b_{n})=x_{k}(\psi\otimes1_{B})(1)=
(\psi\otimes1_{B})(x_{k})=
(\psi\otimes1_{B})(x_{k}\otimes1)=
\psi(x_{k})\otimes1=\psi(x_{k})\in A^{n}$. This shows that $x_{k}b_{i}\in A$ for all $i$ and so $LL'\subseteq A$.
\end{proof}

Note that, under the assumptions of Lemma \ref{Lemma 5 besh}, then $L\otimes_{A}B'\simeq B'$ as $B'$-modules. Indeed, $L\otimes_{A}B'\simeq\phi(L)B'
=\phi(L)A'B'=L_{\phi}B'=B'$.

\begin{corollary}\label{Remark two} Let $\phi:(A,B)\rightarrow(A',B')$ be a morphism in the category of extensions of rings. Then we have the following commutative diagram of Abelian groups: $$\xymatrix{
\mathscr{G}(A,B)\ar[r]\ar[d]&\Pic(A)\ar[d] \\ \mathscr{G}(A',B')\ar[r]&
\Pic(A').}$$
\end{corollary}

\begin{proof} By Lemma \ref{key lemma} and Corollary \ref{Lemma TRRS}, the map $\mathscr{G}(A,B)\rightarrow\Pic(A)$ given by $L\mapsto[L]$ is a morphism of groups.  
If $L\in\mathscr{G}(A,B)$ then by Lemma \ref{Lemma 5 besh}, we have $L\otimes_{A}A'\simeq L_{\phi}$ as $A'$-modules. Thus the above diagram is commutative. 
\end{proof}

Note that Lemmas \ref{Lemma 5 besh} and \ref{key lemma} also improve \cite[Lemmas 2.2, 2.3]{Roberts-Singh}.

\begin{remark}\label{Remark bir} It is important to note that the group of invertible ideals and the ideal class group behave much better with extensions of rings than with arbitrary morphisms of rings. In particular, the notion of an invertible ideal of an extension of rings cannot be generalized to an arbitrary morphism of rings in a satisfactory way that the essential property of invertible ideals, namely projectivity, is preserved. \\ 
More precisely, for a given morphism of rings $f:A\rightarrow B$ we may define an $A$-submodule $L$ of $B$ as an invertible ideal of this map if $LL'=f(A)$ for some $A$-submodule $L'$ of $B$. Then in this case, it can be easily seen that $L$ is a finitely generated $A$-module and $LB=B$. But an invertible ideal of the map $f$ is not necessarily $A$-projective. As an example, consider the canonical ring map $f:\mathbb{Z}\rightarrow\mathbb{Z}/2$, then $f(\mathbb{Z})=\mathbb{Z}/2$ is an invertible ideal of this map but it is not $\mathbb{Z}$-projective (even it is not $\mathbb{Z}$-flat). \\
In addition, it is important to note that every invertible ideal of a ring map $f:A\rightarrow B$ is precisely an invertible ideal of the induced ring extension $f(A)\subseteq B$. Thus the groups of invertible ideals of $f$ and the induced ring extension $f(A)\subseteq B$ are the same. The same goes for the ideal class groups. \\
Hence, in practice, there is no advantage in generalizing invertible ideals to arbitrary ring maps. It is therefore logical to concentrate our investigation mainly on the extensions of the rings only, as we will do in the article.
\end{remark}

We are now ready to prove the first main result of this article:

\begin{theorem}\label{Thm TRS 2} For any extension of rings $A\subseteq B$ we have the following exact sequence of Abelian groups: $$\xymatrix{0\ar[r]&
\mathfrak{C}(A,B)\ar[r]&\Pic(A)\ar[r]&\Pic(B).}$$
\end{theorem}

\begin{proof} We observed that the map $\mathscr{G}(A,B)\rightarrow\Pic(A)$ given by $L\mapsto[L]$ is a morphism of groups. It is not hard to see that $H=\{Ax: x \in B^{\ast}\}$ is the kernel of this map. Indeed, if $x\in B^{\ast}$ then the map $A\rightarrow Ax$ given by $a\mapsto ax$ is an isomorphism of $A$-modules, and so $[Ax]=[A]$ is the identity element of the group $\Pic(A)$. Conversely, assume $L$ is an invertible ideal of the extension $A\subseteq B$ and $h:A\rightarrow L$ is an isomorphism of $A$-modules. Then $L=Ax$ where $x=h(1)$. By Lemma \ref{Lemma pii 20}, $x\in B^{\ast}$. Then $L=Ax\in H$. 
Hence, the map $\mathscr{G}(A,B)\rightarrow\Pic(A)$ induces an injective morphism of groups $\phi:\Cl(A,B)\rightarrow\Pic(A)$. \\ 
We know that the map $\psi=\Pic(j):\Pic(A)\rightarrow\Pic(B)$ given by $L\mapsto L\otimes_{A}B$ is a morphism of groups where $j:A\rightarrow B$ is the inclusion ring map. If $L$ is an invertible ideal of the extension $A\subseteq B$, then by Corollary \ref{Lemma TRRS}, $L\otimes_{A}B\simeq LB=B$ as $A$-modules. It is indeed an isomorphism as $B$-modules, and so $[L\otimes_{A}B]=[B]$ is the identity element of the group $\Pic(B)$. This shows that $\Ima(\phi)\subseteq\Ker(\psi)$. \\
Conversely, take $L\in\Ker(\psi)$. Then $L\otimes_{A}B\simeq B$ as $B$-modules. Thus we have an isomorphism of $B$-modules $f:L\otimes_{A}B\rightarrow B$. Since $L$ is $A$-flat, the canonical morphism of $A$-modules $g:L\rightarrow L\otimes_{A}B$ given by $x\mapsto x\otimes1$ is injective. Then $L$ is isomorphic to $L'=\Ima(fg)$ as $A$-modules. Thus the $A$-submodule $L'$ of $B$ is a finitely generated projective $A$-module of rank 1. We also have $L'B=B$. Indeed, if $b\in B$ then we may write $b=f(\sum\limits_{k=1}^{n}x_{k}\otimes b_{k})=\sum\limits_{k=1}^{n}f(x_{k}\otimes b_{k})=\sum\limits_{k=1}^{n}
f\big(b_{k}\cdot(x_{k}\otimes1)\big)$ where $x_{k}\in L$ and $b_{k}\in B$ for all $k$. But the map $f$ is a morphism of $B$-modules (i.e., it is $B$-linear), thus $b=\sum\limits_{k=1}^{n}
b_{k}f(x_{k}\otimes1)=\sum\limits_{k=1}^{n}
b_{k}fg(x_{k})\in L'B$. Then by Lemma \ref{key lemma}, $L'$ is an invertible ideal of the extension $A\subseteq B$. This shows that $\Ker(\psi)\subseteq\Ima(\phi)$.
This completes the proof. 
\end{proof}

As an immediate consequence, for any extension of rings $A\subseteq B$ then $\Cl(A,B)=0$ if and only if the natural group morphism $\Pic(A)\rightarrow\Pic(B)$ is injective.

We say that an (injective) morphism of rings $f: A\rightarrow B$ has a retraction if there exists a morphism of rings $g:B\rightarrow A$ such that $gf:A\rightarrow A$ is the identity map. For example, for any ring $A$, the ring extensions $A\subseteq A[x]$ and $A\subseteq A[x,x^{-1}]$  have retractions. 
As another example, for any ring map $R\rightarrow S$, then the canonical ring extension $S\rightarrow S\otimes_{R}S$ given by $s\mapsto s\otimes1$ has a retraction. It fact, the canonical ring map $S\otimes_{R}S\rightarrow S$ given by $s\otimes s'\mapsto ss'$ is its retraction. The ring extension $S\rightarrow S\otimes_{R}S$ given by $s\mapsto 1\otimes s$ has also the same retraction.

\begin{corollary}\label{Lemma TH 2} If an extension of rings $A\subseteq B$ has a retraction, then  $\Cl(A,B)=0$. 
\end{corollary}

\begin{proof} There exists a ring map $g:B\rightarrow A$ such that $gf:A\rightarrow A$ is the identity map where $f:A\rightarrow B$ is the inclusion map. Then by the functorial property of the Picard group,  $\Pic(gf)=\Pic(g)\circ\Pic(f):\Pic(A)\rightarrow\Pic(A)$ is the identity map. This shows that the map $\Pic(f):\Pic(A)\rightarrow\Pic(B)$ is injective. Then by Theorem \ref{Thm TRS 2}, we have $\Cl(A,B)=0$.  
\end{proof}

\begin{corollary} If $A$ is a ring and $M$ is an Abelian group or more generally a commutative monoid, then $\Cl(A, A[M])=0$.
\end{corollary}

\begin{proof} The extension of rings $A\subseteq A[M]$ has a retraction. Indeed, by the universal property of monoid rings, the
(augmentation) map $A[M]=\bigoplus\limits_{x\in M}A\rightarrow A$ given by $\sum\limits_{x\in M}r_{x}\epsilon_{x} \mapsto\sum\limits_{x\in M}r_{x}$ is a morphism of rings where $\epsilon_{x}=(\delta_{x,y})_{y\in M}$ and $\delta_{x,y}$ is the Kronecker delta.
Then by Corollary \ref{Lemma TH 2}, $\Cl(A, A[M])=0$.
\end{proof}

In particular, for any ring $A$ we have $\Cl(A, A[x])=\Cl(A, A[x,x^{-1}])=0$. It can be also seen that the extension of rings $A\subseteq A[[x]]$ has a retraction and so $\Cl(A, A[[x]])=0$.  

\begin{example} We show with an example that Corollary \ref{Lemma TH 2} is in its full generality. This means that it cannot be further generalized to split ring maps. By a split ring map we mean an (injective) morphism of rings $f:A\rightarrow B$ that splits as $A$-module (i.e., there exists a morphism of $A$-modules $g:B\rightarrow A$ such that $gf$ is the identity map). It is clear that split extensions are natural generalizations of ring extensions which have retractions. This example also shows that there are extensions of rings $A\subseteq B$ such that $B$ is a free $A$-module and has a basis containing the unit element 1 (such extensions are special cases of split extensions), but
$\Cl(A,B)\neq0$. The ring extension $A=\mathbb{R}[x,y]/(x^{2}+y^{2}-1)\subseteq B=\mathbb{C}[x,y]/(x^{2}+y^{2}-1)$ is split. In fact, $B\simeq\mathbb{C}\otimes_{\mathbb{R}}A$ is a free $A$-module on basis $1, i$. If $u=x+iy$ and $v=x-iy$ then $B=\mathbb{C}[u,v]/(uv-1)\simeq\mathbb{C}[u,u^{-1}]$ is a UFD and so $\Pic(B)=0$. Then by Theorem \ref{Thm TRS 2}, $\Cl(A,B)\simeq\Pic(A)$. But it is well known that $\Pic(A)$ is isomorphic to the additive group $\mathbb{Z}/2$.
\end{example}  

\begin{corollary}\label{coro TrS 2025} For any extension of rings $A\subseteq B$ we have the following exact sequence of Abelian groups: $$\xymatrix{0\ar[r]&B^{\ast}/A^{\ast}\ar[r]&
\mathscr{G}(A,B)\ar[r]&\Pic(A)\ar[r]&\Pic(B).}$$
\end{corollary}

\begin{proof} It follows from Theorem \ref{Thm TRS 2}.
\end{proof}

Let $A$ be a ring. If we take $B=T(A)$, the total ring of fractions of $A$, then every invertible ideal of this extension is simply called an invertible ideal of $A$, and the group
$\Cl(A,B)$ is also called the ideal class group of $A$ and is denoted by $\Cl(A)$. It is easy to see that every invertible ideal $L$ of $A$ is a fractional ideal of $A$. This means that $L$ is an $A$-submodule of $T(A)$ such that $sL\subseteq A$ for some non-zero-divisor $s\in A$. In particular, the classical ideal class group (i.e., when $A$ is an integral domain) is naturally recovered from this construction. 

\begin{corollary}\label{Coro 1 new T} For any ring $A$ we have the following exact sequence of Abelian groups $\xymatrix{0\ar[r]&\mathfrak{C}(A)\ar[r]&\Pic(A)\ar[r]
&\Pic(T(A)).}$
\end{corollary}

\begin{proof} It is clear from Theorem \ref{Thm TRS 2}.
\end{proof}

The sequences in Theorem \ref{Thm TRS 2} and Corollary \ref{Coro 1 new T}, in general, are not right exact (see Example \ref{Example one}).

\section{Geometric aspects of invertible modules}

In this section, we first prove the following important result, the proof of which is based entirely on geometrical ideas:

\begin{theorem}\label{Lemma TD 6} Every commutative ring has a faithfully flat ring extension whose Picard group is trivial.
\end{theorem}

\begin{proof} Let $A$ be a ring. If $L\in\Pic(A)$ then 
there are finitely many elements $f_{1},\ldots, f_{n}\in A$ which generate the unit ideal of $A$ and $L_{f_{i}}\simeq A_{f_{i}}$ as $A_{f_{i}}$-modules for all $i$. Then consider the direct product ring $A_L:=A_{f_{1}}\times\cdots\times A_{f_{n}}$ which is a faithfully flat ring extension of $A$, and $L$ is vanished under the canonical group map $\Pic(A)\rightarrow\Pic(A_L)$. \\ 
Therefore $L$ is also vanished in every extension ring of $A_L$. 
Thus every finite subset $J=\{L_{1},\ldots, L_{n}\}$ of $\Pic(A)$ is contained in the kernel of the canonical group map $\Pic(A)\rightarrow\Pic(R_J)$ induced by the faithfully flat ring extension $A\subseteq  R_J:=\bigotimes\limits_{L\in J}A_L=A_{L_{1}}\otimes_{A}
\cdots\otimes_{A}A_{L_{n}}$. 

Note that if $J$ is the empty set then $R_J=A$, and  if $J\subseteq J'$ are finite subsets of $\Pic(A)$ then we have a canonical ring map $f_{J,J'}:R_J\rightarrow R_{J'}$. Clearly the rings $R_J$, together with the maps $f_{J,J'}$ as transition morphisms, is a direct system of rings over the directed poset of all finite subsets of $\Pic(A)$. Let $A_1=\lim\limits_{\overrightarrow{J}}R_J$ be the direct limit of this system. Therefore, for any ring $A$ we can find a faithfully flat ring extension $A_1$ of $A$ such that the canonical group map $\Pic(A)\rightarrow\Pic(A_{1})$ is zero. \\
Then by induction, we obtain a chain of (faithfully flat) extensions of rings $A_{0}=A\subseteq A_{1}\subseteq A_{2}\subseteq\cdots$ such that for every natural number $n\geqslant0$ the canonical group map $\Pic(A_{n})\rightarrow\Pic(A_{n+1})$ is zero. Since all the transition morphisms $\Pic(A_{m})\rightarrow\Pic(A_{n})$ are zero for $m<n$, we conclude that the direct limit of this system is trivial, i.e., $\lim\limits_{\overrightarrow{n\geqslant0}}\Pic(A_n)=0$.
Setting $B=\lim\limits_{\overrightarrow{n\geqslant0}}A_n$ which is a faithfully flat ring extension of $A$. Since the Picard group commutes with direct limits, we get that $\Pic(B)\simeq\lim\limits_{\overrightarrow{n\geqslant0}}
\Pic(A_n)=0$. This completes the proof. 
\end{proof}

The Picard group can be realized as the ideal class group: 

\begin{corollary}\label{Coro culmination} The Picard group of every commutative ring is canonically isomorphic to the ideal class group of a faithfully flat extension of that ring. 
\end{corollary}

\begin{proof} It follows from Theorems \ref{Thm TRS 2} and \ref{Lemma TD 6}. As an alternative proof, in the proof of Theorem \ref{Lemma TD 6}, we observed that for any ring $A$ then there exists a faithfully flat extension ring $A_1$ of $A$ such that the induced map $\Pic(A)\rightarrow\Pic(A_{1})$ is zero. Then by Theorem \ref{Thm TRS 2}, we have $\Pic(A)\simeq\mathfrak{C}(A,A_1)$.
\end{proof}

Next, we prove another useful result on invertible ideals:

\begin{theorem}\label{lemma toofan 20} Every invertible ideal of an extension of rings $A\subseteq B$ with $A$ has finitely many maximal ideals is principal. 
\end{theorem}

\begin{proof} Take $L\in\mathscr{G}(A,B)$. Let $M_{1},\ldots, M_{n}$ be all (distinct) maximal ideals of $A$. For each $k$, since $M_{k}\neq A=LL^{-1}$, we may choose some $x_{k}\in L$ and $y_{k}\in L^{-1}$ such that $x_{k}y_{k}\notin M_{k}$. Also, for each $k$ there exists some $a_{k}\in\bigcap\limits_{\substack{1\leqslant i\leqslant n,\\i\neq k}}M_{i}$ such that $a_{k}\notin M_{k}$. We have $y:=\sum\limits_{i=1}^{n}a_{i}y_{i}\in L^{-1}$ and $L(Ay)=Ly$ is an ideal of $A$. Then we show  that $Ly=A$. If not, then $Ly\subseteq M_{k}$ for some $k$. It follows that $x_{k}y=(x_{k}y_{k})a_{k}+
\sum\limits_{\substack{1\leqslant i\leqslant n,\\i\neq k}}(x_{k}y_{i})a_{i}\in M_{k}$. But each $x_{k}y_{i}\in LL^{-1}=A$ and so $\sum\limits_{\substack{1\leqslant i\leqslant n,\\i\neq k}}(x_{k}y_{i})a_{i}\in M_{k}$. This yields that $(x_{k}y_{k})a_{k}\in M_{k}$ which is a contradiction. Therefore $Ly=A$. It follows that  $L^{-1}=Ay$, because $\mathscr{G}(A,B)$ is a group. By Lemma \ref{Lemma pii 20}, $y\in B^{\ast}$. Thus $L=Ay^{-1}$ is a principal invertible ideal of the extension $A\subseteq B$.  
\end{proof}

Recall that if $M$ is a module over a ring $A$, then, up to canonical isomorphisms (i.e., no matter how one puts in parentheses, by the associativity of tensor, they are assumed identical), we can define $M^{\otimes0}=A$, $M^{\otimes1}=M$, $M^{\otimes2}=M\otimes_{A}M$ and so on. In fact, $M^{\otimes(n+1)}=M^{\otimes n}\otimes_{A}M$  for all $n\geqslant0$. We also define $M^{\otimes(-n)}:=(\widehat{M})^{\otimes n}$ for all $n\geqslant1$ where $\widehat{M}=\Hom_{A}(M,A)$ is the dual of $M$. Then we have the following result:

\begin{lemma}\label{Lemma 7 nf} Let $L$ be an invertible module over a ring $A$. If  $f:L\rightarrow A$ is a morphism of $A$-modules, then $f(x)y=f(y)x$ for all $x,y\in L$. In addition, $\bigoplus\limits_{n\in\mathbb{Z}}L^{\otimes n}$ is a $\mathbb{Z}$-graded commutative ring. 
\end{lemma}

\begin{proof} We know that there are (finitely) many elements $b_{1},\ldots, b_{n}\in A$ which generate the unit ideal of $A$ and $L_{b}$ is a free $A_{b}$-module of constant rank 1  for all $b\in\{b_{1},\ldots, b_{n}\}$. Thus each $L_{b}=A_{b}(z_{b})$ for some $z_{b}\in L_{b}$. 
Then in $L_{b}$ we have $x/1=rz_{b}$ and $y/1=r'z_{b}$ for some $r,r'\in A_{b}$. We know that the map $f_{b}:L_{b}\rightarrow A_{b}$, induced by $f$, is a morphism of $A_{b}$-modules. Thus 
$f_{b}(x/1)y/1=rr'f_{b}(z_{b})z_{b}=f_{b}(y/1)x/1$.
This shows that the image of $f(x)y-f(y)x$ under the canonical map $L\rightarrow L_{b}$ is zero. Thus there exists a natural number $N\geqslant1$ such that $b^{N}(f(x)y-f(y)x)=0$ for all $b\in\{b_{1},\ldots, b_{n}\}$. But we may write $1=\sum\limits_{k=1}^{n}a_{k}b^{N}_{k}$ where $a_{1},\ldots, a_{n}\in A$. Then $f(x)y-f(y)x=\sum\limits_{k=1}^{n}a_{k}
b^{N}_{k}((f(x)y-f(y)x))=0$. \\
The canonical isomorphism of $A$-modules $\widehat{L}\otimes_{A}L\rightarrow A$ induces a (unique) operation $\ast$ (called multiplication) on the $A$-module $B:=\bigoplus\limits_{n\in\mathbb{Z}}L^{\otimes n}$ such that $f\ast x=f(x)$ for all $f\in\widehat{L}$ and $x\in L$, and makes $B$ into a $\mathbb{Z}$-graded ring with homogeneous components $L^{\otimes n}$. This ring also contains the tensor algebras $\bigoplus\limits_{n\geqslant0}L^{\otimes n}$ and $\bigoplus\limits_{n\geqslant0}L^{\otimes(-n)}$ as subrings. To show that the ring $B$ is 
commutative, it suffices to show that the product of any two homogeneous elements is commutative. First we show that $x\otimes y=y\otimes x$ for all $x,y\in L$. We may write $1=\sum\limits_{i=1}^{n}f_{i}(x_{i})$ where $x_{i}\in L$ and $f_{i}\in \widehat{L}$ for all $i$. Then $x\otimes y=\big(\sum\limits_{i=1}^{n}f_{i}(x_{i})\big)x\otimes y=\big(\sum\limits_{i=1}^{n}f_{i}(x_{i})x\big)\otimes y=\big(\sum\limits_{i=1}^{n}f_{i}(x)x_{i}\big)\otimes y=
\sum\limits_{i=1}^{n}x_{i}\otimes f_{i}(x)y=\sum\limits_{i=1}^{n}x_{i}\otimes f_{i}(y)x=\sum\limits_{i=1}^{n}f_{i}(y)x_{i}\otimes x=\sum\limits_{i=1}^{n}f_{i}(x_{i})y\otimes x=y\otimes x$. \\
This shows that for any $L\in\Pic(A)$ the ring $\bigoplus\limits_{n\geqslant0}L^{\otimes n}$, the tensor algebra of $L$, is commutative. Hence, $\bigoplus\limits_{n\geqslant0}L^{\otimes(-n)}$,
the tensor algebra of $\widehat{L}$, is also commutative.  
We also have $f(x)g=g(x)f$ for all $f,g\in \widehat{L}$ and all $x\in L$. Finally, note that for any finitely many elements $x_{1},\ldots,x_{n}\in L$ and $f_{1},\ldots, f_{m}\in\widehat{L}$ with $m,n\geqslant1$ we have $f\ast(x_{1}\otimes\cdots\otimes x_{n})=f(x_{1})x_{2}\otimes\cdots\otimes x_{n}=x_{1}\otimes f(x_{2})x_{3}\otimes\cdots \otimes x_{n}=x_{1}\otimes x_{2}\otimes f(x_{3})x_{4}\otimes\cdots\otimes x_{n}=\ldots=x_{1}\otimes x_{2}\otimes\cdots\otimes f(x_{n})x_{n-1}=(x_{1}\otimes\cdots\otimes x_{n})\ast f$ and so $(f_{1}\otimes\cdots\otimes f_{m})\ast(x_{1}\otimes\cdots\otimes x_{n})=(x_{1}\otimes\cdots\otimes x_{n})\ast(f_{1}\otimes\cdots\otimes f_{m})$. This completes the proof.  
\end{proof}

If $R=\bigoplus\limits_{n\geqslant0}R_{n}$ is an $\mathbb{N}$-graded ring then the ring extension $R_{0}\subseteq R$ has a retraction, because the map $R\rightarrow R_{0}$ given by $\sum\limits_{n\geqslant0}r_{n}\mapsto r_{0}$ is a ring morphism (note that the $\mathbb{N}$-grading is necessary for this map to be a ring morphism). Then by Corollary \ref{Lemma TH 2}, we have $\Cl(R_{0}, R)=0$. 

In particular, if $L$ is an invertible module over a ring $A$, then its tensor algebra $B:=\bigoplus\limits_{n\geqslant0}L^{\otimes n}$ is an $\mathbb{N}$-graded commutative ring (see Lemma \ref{Lemma 7 nf}). So $\Cl(A, B)=0$. 

We give a new proof to the following well known result:

\begin{corollary}\label{coro toofan 21} The Picard group of every ring with finitely many maximal ideals is trivial. 
\end{corollary}

\begin{proof} Let $L$ be an invertible module over a ring $A$. Then consider $B=\bigoplus\limits_{n\in\mathbb{Z}}L^{\otimes n}$ which is a commutative ring extension of $A$ (see Lemma \ref{Lemma 7 nf}). It can be easily seen that in this ring $B$ we have $L\widehat{L}=A$. This shows that $L$ is an invertible ideal of the extension  $A\subseteq B$.  Therefore if $A$ has finitely many maximal ideals then by Theorem \ref{lemma toofan 20}, $L$ is principal, i.e., $L=Ab$ for some $b\in B^{\ast}$. Hence, $L\simeq A$ as $A$-modules.   
\end{proof}

The Picard group does not necessarily commute with the ``infinite" direct product of rings. However, we have the following result:

\begin{theorem}\label{key lemma 2 Pic} Let $(R_i)$ be a family of rings. Then the group $\Pic(\prod\limits_{i}R_{i})$ can be canonically embedded in the group $\prod\limits_{i}\Pic(R_{i})$. 
\end{theorem}

\begin{proof} We show that the natural morphism of groups $\Pic(\prod\limits_{i}R_{i})\rightarrow
\prod\limits_{i}\Pic(R_{i})$ given by $L\mapsto(L\otimes_{R}R_{i})_{i}$ is injective. To see this, it suffices to show that every invertible  $R$-module and more generally every finitely presented $R$-module $L$ is isomorphic to $\prod\limits_{i}L_{i}$ where $L_{i}=L\otimes_{R}R_{i}$ (recall that every finitely generated projective module is finitely presented). In fact, we show that the natural morphism of $R$-modules $\phi:L\rightarrow\prod\limits_{i}L_i$ given by $x\mapsto(x\otimes1)_{i}$ is an isomorphism.
Indeed, first tensoring a finite presentation: $$\xymatrix{R^{n}\ar[r]&R^{m}\ar[r]&L\ar[r]&0}$$ of $L$ with $-\otimes_{R}R_{i}$ to get a finite presentation for $L_i$: 
$$\xymatrix{R_{i}^{n}\ar[r]&R_{i}^{m}\ar[r]&L_{i}
\ar[r]&0.}$$ 
Next taking direct products, the following exact sequence of $R$-modules is obtained: 
$$\xymatrix{(\prod\limits_{i}R_{i}^{n})
\ar[r]&(\prod\limits_{i}R_{i}^{m})\ar[r]&
(\prod\limits_{i}L_{i})\ar[r]&0.}$$
Then by using the fact that for any family of modules $(N_k)$ over a ring $R$, the $R$-modules $(\prod\limits_{k}N_{k})^{\oplus n}$ and $\prod\limits_{k}N_{k}^{\oplus n}$ are canonically isomorphic for all natural numbers $n\geqslant0$, we get the following commutative diagram of $R$-modules with exact rows:  
$$\xymatrix{R^{n}\ar[r]\ar[d]^{\simeq}&
R^{m}\ar[r]\ar[d]^{\simeq}&L\ar[r]\ar[d]^{\phi}&0 \\ (\prod\limits_{i}R_{i}^{n})
\ar[r]&(\prod\limits_{i}R_{i}^{m})\ar[r]&
(\prod\limits_{i}L_{i})\ar[r]&0.}$$
So by the five lemma, the map $\phi$ is an isomorphism of $R$-modules.   
\end{proof}

\begin{corollary}\label{Thm toofan T20 1} Let $(R_i)$ be a family of rings and $R=\prod\limits_{i}R_i$. Then $\Pic(R)=0$ if and only if $\Pic(R_i)=0$ for all $i$. 
\end{corollary}

\begin{proof} First note that for any idempotent $e$ of a commutative ring $R$, then $Re$ is a commutative ring with the identity element $e$, and we have a decomposition of rings $R\simeq Re \times R(1-e)$. Thus $\Pic(R)\simeq\Pic(Re)\times \Pic(R(1-e))$. Each factor ring $R_k$ is isomorphic to $R/(1-e) \simeq Re$ for some idempotent $e=(\delta_{i,k})_{i} \in R$ where $\delta_{i,k}$ is the Kronecker delta. Now if $\Pic(R)=0$, then $\Pic(R_k)\simeq\Pic(Re)=0$ for all $k$. The reverse implication follows from Theorem \ref{key lemma 2 Pic}. 
\end{proof}

\begin{corollary}\label{Coro Pic 1 TDCR} The Picard group of every direct product of local rings is trivial. 
\end{corollary}

\begin{proof} Apply Corollaries \ref{coro toofan 21} and \ref{Thm toofan T20 1}.
\end{proof}

A weak version of Theorem \ref{Lemma TD 6} can be proved in an elementary fashion as follows: 

\begin{corollary}\label{coro nice 3046782 n} Every ring has a ring extension whose Picard group is trivial. 
\end{corollary}

\begin{proof} Every ring $R$ can be canonically embedded in $A:=\prod\limits_{M\in\Max(R)}R_{M}$. By Corollary \ref{Coro Pic 1 TDCR}, we have $\Pic(A)=0$.
\end{proof}

\begin{corollary}\label{Corollary Pic iso Cl} Let $A\subseteq B$ be an extension of rings such that $B$ has finitely many maximal ideals. Then we have the canonical isomorphism of groups $\Cl(A,B)\simeq\Pic(A)$.
\end{corollary}

\begin{proof} Apply Theorem \ref{Thm TRS 2} and Corollary \ref{coro toofan 21}.
\end{proof}

\begin{corollary}\label{Coro rslrpc} If $A$ is a reduced ring with finitely many minimal primes, then we have the canonical isomorphism of groups $\Cl(A)\simeq\Pic(A)$.
\end{corollary}

\begin{proof} Since $A$ is reduced, the set of its zero-divisors is the union of its minimal prims, i.e.,  $Z(A)=\bigcup\limits_{\mathfrak{p}\in\Min(A)}\mathfrak{p}$. Then using the prime avoidance lemma (see e.g. \cite[Theorem 2.1]{Tarizadeh-Chen} or \cite[Theorem 2.2]{Tarizadeh-Chen 2}), we conclude that $T(A)$ has finitely many maximal ideals. Then by Corollary \ref{coro toofan 21}, the Picard group of $T(A)$ is trivial. Then the assertion is deduced from Corollary \ref{Coro 1 new T}.
\end{proof}

As an immediate consequence of Corollary \ref{Coro rslrpc}, for every integral domain $A$, we have the canonical isomorphism of groups $\Cl(A)\simeq\Pic(A)$. In particular, every Abelian group is the Picard group of some integral domain (because every Abelian group is the ideal class group of some Dedekind domain \cite[Theorem 7]{Claborn}). 

\begin{example}\label{Example one} The sequence in Theorem \ref{Thm TRS 2} is not right exact in general. For instance, take an integral domain $R$ whose ideal class group is non-trivial. Thus $\Pic(R)\simeq\Cl(R)\neq0$.
But we have an extension of rings $\mathbb{Z}\subseteq R$ or $\mathbb{Z}/p\subseteq R$ for some prime number $p$. We know that the Picard group of every UFD is trivial, and so $\Pic(\mathbb{Z})=\Pic(\mathbb{Z}/p)=0$.    
\end{example}

The following result and Corollary \ref{Coro JPAA} show that invertible modules have the avoidance property: 

\begin{lemma}\label{Lemma avoidance} Let $A\subseteq B$ be an extension of rings and $L$ be an invertible ideal of this extension. If $L_{1},\cdots, L_{n}$ are finitely many $A$-submodules of $B$ with $L\subseteq\bigcup\limits_{k=1}^{n}L_{k}$, then $L\subseteq L_{k}$ for some $k$.
\end{lemma}

\begin{proof} We may write $L=\bigcup\limits_{k=1}^{n}L'_{k}$ where $L'_{k}=L\cap L_{k}$ for all $k$. Since $LL^{-1}=A$, each $L'_{k}=AL'_{k}=L(L^{-1}L'_{k})$ and $L^{-1}L'_{k}\subseteq L^{-1}L=A$ is an ideal of $A$ (in the usual sense). Then $L=\bigcup\limits_{k=1}^{n}(L^{-1}L'_{k})L$.
We know that $L$ is a finitely generated faithful $A$-module (faithfulness means that $\Ann(L)=0$). It is well known that if $M$ is a finitely generated faithful module over a ring $A$ and $I_{1},\ldots, I_{n}$ are finitely many ideals of $A$ with $M=\bigcup\limits_{k=1}^{n}I_{k}M$, then $I_{k}=A$ for some $k$ (see \cite[Lemma 3.1]{Tarizadeh-Chen}). Therefore, $L^{-1}L'_{k}=A$ and so $L=L'_{k}\subseteq L_{k}$ for some $k$.  
\end{proof} 

\begin{corollary}\label{Coro JPAA} Let $L$ be an invertible module over a ring $A$. If $L_{1},\cdots, L_{n}$ are finitely many $A$-submodules of $L$ with $L=\bigcup\limits_{k=1}^{n}L_{k}$, then $L=L_{k}$ for some $k$. 
\end{corollary}

\begin{proof} By Theorem \ref{Lemma TD 6} or Corollary \ref{coro nice 3046782 n}, there exists an extension ring $B$ of $A$ such that $\Pic(B)=0$. Then by Theorem \ref{Thm TRS 2}, there exists an invertible ideal $L'$ of the extension $A\subseteq B$ such that $L\simeq L'$ as $A$-modules. Then by Lemma \ref{Lemma avoidance}, $L=L_{k}$ for some $k$. 
\end{proof} 

The following result slightly improves the related result in the literature (that is, $M$ is not necessarily to be projective): 

\begin{lemma}\label{Lemma simproved 20} Let $M$ be a finitely generated flat module and $N$ be a finitely generated projective module over a ring $R$ whose rank maps are the same. Then every surjective morphism of $R$-modules $f: M\rightarrow N$ is an isomorphism. 
\end{lemma}

\begin{proof} Since $N$ is projective, we have a decomposition of $R$-modules $M\simeq N\oplus K$ where $K=\Ker(f)$. Since $M$ is a finitely generated flat $R$-module, $K$ is as well. Then, between the rank maps, we have the relation $\mathrm{r}_{M}=\mathrm{r}_{N}+\mathrm{r}_{K}$. Thus by hypothesis, $\mathrm{r}_{K}=0$. This shows that $K_{\mathfrak{p}}=0$ for all $\mathfrak{p}\in\Spec(R)$. Hence $K=0$. 
\end{proof}

\begin{lemma}\label{Remark 10 incredible 10} Let $e_{1},\ldots,e_{m}$ be idempotents of a ring $R$ such that  $\Spec(R)=\bigcup\limits_{i=1}^{m}D(e_{i})$. Then there are orthogonal idempotents $g_{1},\ldots,g_{d}\in R$ ($d\leqslant 2^{m}-1$) such that for each $k$ there exists some $i$ with $g_{k}=g_{k}e_{i}$ and $\sum\limits_{k=1}^{d}g_{k}=1$. 
\end{lemma}

\begin{proof} For each nonempty subset $S$ of $\{1,\ldots, m\}$, consider the idempotent $g_{S}:=(\prod\limits_{i\in S}e_{i})(\prod\limits_{j\notin S}(1-e_{j}))$. Clearly the idempotents $g_{S}$ are orthogonal and $g_{S}=g_{S}e_{i}$ for $i\in S$. We claim that $\sum\limits_{S}g_{S}=1$ where $S$ ranges in the set of nonempty subsets of $\{1,\ldots, m\}$. Since the idempotents $g_{S}$ are orthogonal, the element $\sum\limits_{S}g_{S}$ is idempotent. Then to prove $\sum\limits_{S}g_{S}=1$, it suffices to show that $D(\sum\limits_{S}g_{S})=D(1)$. For any idempotents $e,e'\in R$ we have $D(e)\cup D(e')=D(e+e'-ee')$. Using this and the orthogonality of the $g_k$, we get that  $D(\sum\limits_{S}g_{S})=\bigcup\limits_{S}D(g_{S})$.
Suppose there is some $\mathfrak{p}\in D(1)=\Spec(R)$ such that $\mathfrak{p}\notin\bigcup\limits_{S}D(g_{S})$. Then for each nonempty subset $S$ of $\{1,\ldots, m\}$ we have $g_{S}\in\mathfrak{p}$. But $e_{k}\notin\mathfrak{p}$ for some $k$.
For $S=\{k\}$ we have $g_{S}=e_{k}(\prod\limits_{j\neq k}(1-e_{j}))\in\mathfrak{p}$. So there exists some index $\ell$ with $\ell\neq k$ such that $1-e_{\ell}\in\mathfrak{p}$. Then take $S=\{k,\ell\}$, we have $g_{S}=e_{k}e_{\ell}(\prod\limits_{j\neq k,\ell}(1-e_{j}))\in\mathfrak{p}$ and so $1-e_{j}\in\mathfrak{p}$ for some $j\neq k,\ell$. Continuing this process, we conclude that for $S=\{1,\ldots, m\}$, we have $g_{S}=\prod\limits_{i=1}^{m}e_{i}\notin\mathfrak{p}$ which is a contradiction. 
\end{proof}

\begin{theorem}\label{Thm zero dim} The Picard group of every zero-dimensional ring is trivial.
\end{theorem}

\begin{proof} Let $R$ be a zero-dimensional ring. By Theorem \ref{Thm 100 npf}, we have a canonical isomorphism of groups $\Pic(R)\simeq\Pic(R_{\mathrm{red}})$. We know that $R_{\mathrm{red}}$ is a von Neumann regular (absolutely flat) ring. Hence, to prove the assertion, it suffices to show that the Picard group of every von Neumann regular ring $R$ is trivial. Let $L$ be an invertible $R$-module. So the canonical morphism of $R$-modules $\widehat{L}\otimes_{R}L\rightarrow R$ is an isomorphism. Then we may write $1=\sum\limits_{i=1}^{m}f_{i}(x_{i})$ where $f_{i}\in\widehat{L}$ and $x_{i}\in L$ for all $i$. Since $R$ is von Neumann, for each $a\in R$ there exists some $b\in R$ such that $a=a^{2}b$. Then $e=ab$ is idempotent and $Ra=Re$. Therefore, for each $i$, there exists an idempotent $e_{i}\in R$ such that $f_{i}(x_{i})R=e_{i}R$ and $\Spec(R)=\bigcup\limits_{i=1}^{m}D(e_{i})$. By Lemma \ref{Remark 10 incredible 10}, there are orthogonal idempotents $g_{1},\ldots,g_{d}\in R$ such that for each $k$ there exists some $i$ for which $g_{k}=g_{k}e_{i}$ and $\sum\limits_{k=1}^{d}g_{k}=1$. Clearly the map $h:g_{k}L\rightarrow g_{k}R$ given by $g_{k}x\mapsto g_{k}f_{i}(x)$ is a morphism of $g_{k}R$-modules (note that $g_{k}R$ is a ring with the identity $g_{k}$). This map is surjective, because $g_{k}r=g_{k}e_{i}r=g_{k}f_{i}(x_{i})r'=
g_{k}f_{i}(r'x_{i})$. Clearly $g_{k}L\simeq L/(1-g_{k})L
\simeq L\otimes_{R}R/(1-g_{k})$ is an invertible module over $g_{k}R\simeq R/(1-g_{k})$. Then by Lemma \ref{Lemma simproved 20}, $h$ is an isomorphism of $g_{k}R$-modules (and so it is an isomorphism of $R$-modules). Since the $g_{k}$ are orthogonal with summing 1, we have the isomorphisms of $R$-modules $L=\sum\limits_{k=1}^{d}g_{k}
L\simeq\bigoplus\limits_{k=1}^{d}g_{k}
L\simeq\bigoplus\limits_{k=1}^{d}g_{k}R\simeq
\sum\limits_{k=1}^{d}g_{k}R=R$. Hence, $\Pic(R)=0$. 
\end{proof}

In particular, Picard groups of finite rings, Artinian rings, and Boolean rings are trivial.

We conclude this section with a problem whose proof is not complete yet. First note that if $L$ is an invertible ideal of a ring extension $A\subseteq B$, then we observed that $L\otimes_{A}B\simeq LB=B$ as $B$-modules. Thus $L$ and hence the subgroup of $\Pic(A)$ generated by $L$ is contained in the kernel of natural group morphism $\Pic(A)\rightarrow\Pic(B)$. 

In particular, if $L$ is an invertible module over a ring $A$ then $L$ is an invertible ideal of the ring extension $A\subseteq B:=\bigoplus\limits_{n\in\mathbb{Z}}L^{\otimes n}$ and so the subgroup of $\Pic(A)$ generated by $L$ is contained in the kernel of natural group morphism $\Pic(A)\rightarrow\Pic(B)$. This inclusion can be strict when $\Spec(A)$ is not connected. So it is natural to ask if this inclusion is equality when $\Spec(A)$ is connected.    

\section{The ideal class group behavior with a tower of extensions}
 
In this section we continue to examine further properties of the ideal class group. In particular, we are interested in knowing how the ideal class group behaves when a tower of ring extensions are involved. In this context, we first come to the following conclusion:

\begin{theorem}\label{Theorem fourth TS} Let $A\subseteq B\subseteq C$ be extensions of rings. Then we have the following exact sequence of Abelian groups:  $$\xymatrix{0\ar[r]&
\mathfrak{C}(A,B)\ar[r]&\mathfrak{C}(A,C)\ar[r]&
\mathfrak{C}(B,C).}$$
\end{theorem}

\begin{proof} It is clear that $\mathscr{G}(A,B)$ is a subgroup of $\mathscr{G}(A,C)$.
The map $f:\mathscr{G}(A,C)\rightarrow\mathscr{G}(B,C)$ given by $L\mapsto LB$ is a morphism of groups. If $L\in\mathscr{G}(A,B)$ then $LB=B$. This shows that $\mathscr{G}(A,B)\subseteq\Ker(f)$. If $L\in\Ker(f)$ then $LB=B$ and so $L\subseteq B$. This shows that $L\in\mathscr{G}(A,B)$. Hence, the following sequence is exact: $$\xymatrix{0\ar[r]&
\mathscr{G}(A,B)\ar[r]&\mathscr{G}(A,C)\ar[r]^{f}&
\mathscr{G}(B,C).}$$ Since $B^{\ast}\subseteq C^{\ast}$, we have $H_{1}=\{Ax: x\in B^{\ast}\}\subseteq H_{2}=\{Ax: x\in C^{\ast}\}$. This induces a group map $g:\Cl(A,B)\rightarrow
\Cl(A,C)$ that is given by $LH_{1}\mapsto LH_{2}$. \\ 
If $L\in H_{2}$ for some $L\in\mathscr{G}(A,B)$ then $L=Ax$ for some $x\in C^{\ast}$. But $L\subseteq B$ and
$Ax^{-1}=L^{-1}\subseteq B$.  
It follows that $x\in B^{\ast}$ and so $L\in H_{1}$. Hence $g$ is injective. \\
It is clear that $f(H_{2})\subseteq H_{3}=\{Bx: x\in C^{\ast}\}$. Then we obtain a group map $h:\Cl(A,C)\rightarrow\Cl(B,C)$ that is given by $LH_{2}\mapsto(LB)H_{3}$. Then it is clear that $\Ima(g)\subseteq\Ker(h)$. \\
If $LH_{2}\in\Ker(h)$ with $L\in\mathscr{G}(A,C)$, then $LB=Bx$ for some $x\in C^{\ast}$. But $L_{1}=\{b\in B: bx\in L\}$ and $L_{2}=\{b\in B: bx^{-1}\in L^{-1}\}$ are $A$-submodules of $B$. We also have $L_{1}L_{2}=A$. Because if $b\in L_{1}$ and $b'\in L_{2}$ then $bb'=(bx)(b'x^{-1})\in LL^{-1}=A$. This shows that $L_{1}L_{2}\subseteq A$. \\
To see the reverse inclusion, we may write $1=\sum\limits_{k=1}^{n}c_{k}c'_{k}$ where $c_{k}\in L$ and $c'_{k}\in L^{-1}$ for all $k$. But from $L\subseteq LB=Bx$ we get that each $c_{k}=b_{k}x$ for some $b_{k}\in B$. Then $b_{k}\in L_{1}$ for all $k$. 
We also have $L^{-1}B=Bx^{-1}$. Thus each $c'_{k}=b'_{k}x^{-1}$ for some $b'_{k}\in B$. Then $b'_{k}\in L_{2}$ for all $k$.  It follows that $1=\sum\limits_{k=1}^{n}b_{k}b'_{k}\in L_{1}L_{2}$. This shows that $A\subseteq L_{1}L_{2}$. Hence, $L_{1}$ is an invertible ideal of the extension $A\subseteq B$, i.e., $L_{1}\in\mathscr{G}(A,B)$. We also have $L=L_{1}x$. This shows that $g(L_{1}H_{1})=L_{1}H_{2}=LH_{2}$. Thus $\Ker(h)\subseteq\Ima(g)$. This completes the proof.  
\end{proof}  

\begin{remark} Here we give an alternative proof of Theorem \ref{Theorem fourth TS}. Although the first proof was quite elementary, this proof contains new ideas and also relies on diagram chasing. 
If $\phi:M'\rightarrow M$ and $\psi:M\rightarrow M''$ are morphisms of groups (or morphisms of modules over a fixed ring), then we have the following exact sequence of groups (or modules): $\xymatrix{0\ar[r]&\Ker(\phi)\ar[r]^{j}&\Ker(\psi\phi)
\ar[r]^{h}&\Ker(\psi)}$ where $j$ is the inclusion map and $h$ is defined as $x\mapsto\phi(x)$. \\
We know that any morphisms of rings $f:A\rightarrow B$ and $g:B\rightarrow C$ induce the morphisms of groups $\Pic(f):\Pic(A)\rightarrow\Pic(B)$ and $\Pic(g):\Pic(B)\rightarrow\Pic(C)$. Hence, we obtain the following exact sequence of groups: $$\xymatrix{0\ar[r]&\Ker(\Pic(f))\ar[r]^{j}&
\Ker(\Pic(gf))
\ar[r]^{h}&\Ker(\Pic(g)).}$$
Now, let $A\subseteq B\subseteq C$ be extensions of rings. Then we have the following diagram of Abelian groups: $$\xymatrix{0\ar[r]&
\mathfrak{C}(A,B)\ar[r]\ar[d]^{\simeq}&
\mathfrak{C}(A,C)\ar[r]
\ar[d]^{\simeq}&\mathfrak{C}(B,C)\ar[d]^{\simeq} \\ 0\ar[r]&\Ker(\Pic(j_{1}))\ar[r]&
\Ker(\Pic(j_{2}j_{1}))\ar[r]&\Ker(\Pic(j_{2}))}$$ where $j_{1}:A\rightarrow B$ and $j_{2}:B\rightarrow C$ are the inclusion maps. In the above diagram, the bottom row is exact (see the above discussion). By Theorem \ref{Thm TRS 2}, the vertical maps are isomorphisms. By Corollary \ref{Remark two}, the above diagram is also commutative. Hence, the top row is exact. 
\end{remark}

We say that a morphism of rings $f:A\rightarrow B$
preserves non-zero-divisors if $f(A\setminus Z(A))\subseteq B\setminus Z(B)$. For example, every extension of integral domains preserves non-zero-divisors. As another example, for any ring $A$, the canonical ring map $A\rightarrow A_{\mathrm{red}}$ preserves non-zero-divisors.

Let $A\subseteq B$ be an extension of rings. Then the extensions of rings $A\subseteq B\subseteq T(B)$ gives us a canonical group morphism $\Cl(A,B)\rightarrow\Cl(B)$.
In general, since the total ring of fractions is not a functorial construction, there is no canonical group morphism from $\Cl(A)$ to $\Cl(B)$. Also, there is no canonical group map from $\Cl(A,B)$ to $\Cl(A)$ (or vice versa). But under some conditions we have the following result: 

\begin{theorem}\label{Coro ilginch uch} If $A$ is a reduced ring with finitely many minimal primes and $A\subseteq B$ is an extension of rings which preserves non-zero-divisors, then we have the following exact sequence of Abelian groups: $$\xymatrix{0\ar[r]&
\mathfrak{C}(A,B)\ar[r]&\mathfrak{C}(A)\ar[r]&
\mathfrak{C}(B).}$$
\end{theorem}

\begin{proof} Since the extension $A\subseteq B$ preserves non-zero-divisors, we obtain an (injective) ring map $f:T(A)\rightarrow T(B)$ that is given by $a/s\mapsto a/s$. It is clear that $f(A)\subseteq B$. Thus $f$ is a morphism in the category of extensions of rings. Hence, we obtain a canonical morphism of groups $g:\mathfrak{C}(A)\rightarrow\mathfrak{C}(B)$. Then by Corollary \ref{Remark two}, the following diagram is commutative: $$\xymatrix{
\mathfrak{C}(A)\ar[r]^{g}\ar[d]&\mathfrak{C}(B)\ar[d] \\ \Pic(A)\ar[r]&\Pic(B).}$$ Then, using this and Theorem \ref{Thm TRS 2}, we obtain a group morphism $\Ker(g)\rightarrow\mathfrak{C}(A,B)$ and so the following diagram is commutative: $$\xymatrix{0\ar[r]&
\Ker(g)\ar[r]\ar[d]&\mathfrak{C}(A)\ar[r]^{g}
\ar[d]^{\simeq}&\mathfrak{C}(B)\ar[d] \\ 0\ar[r]&\mathfrak{C}(A, B)\ar[r]&
\Pic(A)\ar[r]&\Pic(B)}$$
where the rows are exact (for the exactness of the bottom row see Theorem \ref{Thm TRS 2}), the third vertical map $\mathfrak{C}(B)\rightarrow\Pic(B)$ is injective,
and by Corollary \ref{Coro rslrpc}, the second vertical map $\mathfrak{C}(A)\rightarrow\Pic(A)$ is an isomorphism. Then by the five lemma, the first vertical map $\Ker(g)\rightarrow\mathfrak{C}(A,B)$ is an isomorphism. This completes the proof.  
\end{proof}

\begin{corollary} If $A\subseteq B$ is an extension of integral domains, then we have the following exact sequence of Abelian groups: $$\xymatrix{0\ar[r]&
\mathfrak{C}(A,B)\ar[r]&\mathfrak{C}(A)\ar[r]&
\mathfrak{C}(B).}$$ 
\end{corollary}

\begin{proof} It follows from Theorem \ref{Coro ilginch uch}.
\end{proof}

The next main result of this article reads as follows:

\begin{theorem}\label{theorem TH} For any extension of rings $A\subseteq B$, the canonical morphism of groups $\Cl(A,B)\rightarrow\Cl(A_{\mathrm{red}},
B_{\mathrm{red}})$ is an isomorphism. 
\end{theorem}

\begin{proof} The ring extension $A\subseteq B$ induces the ring extension $A_{\mathrm{red}}\subseteq B_{\mathrm{red}}$, and the canonical ring map $f:B\rightarrow B_{\mathrm{red}}$ maps $A$ onto $A_{\mathrm{red}}$. This shows that $f$ is a morphism in the category of extensions of rings and so we obtain a morphism of groups $\Cl(A,B)\rightarrow\Cl(A_{\mathrm{red}}, B_{\mathrm{red}})$. Then consider the following diagram: $$\xymatrix{0\ar[r]&
\mathfrak{C}(A,B)\ar[r]\ar[d]&\Pic(A)\ar[r]\ar[d]^{\simeq}&
\Pic(B)\ar[d]^{\simeq} \\ 0\ar[r]&\mathfrak{C}(A_{\mathrm{red}},
B_{\mathrm{red}})\ar[r]&\Pic(A_{\mathrm{red}})
\ar[r]&\Pic(B_{\mathrm{red}})}$$
where the rows are exact (see Theorem \ref{Thm TRS 2}) and the second and third vertical maps are the canonical isomorphisms (see Theorem \ref{Thm 100 npf}). The above diagram is also commutative, since the commutativity of the left-hand side follows from Corollary \ref{Remark two}, and the right-hand side is clear. Then by the five lemma, the first vertical map is an isomorphism.  
\end{proof} 

For a given ring $A$, if $s$ is a non-zero-divisor of $A$ then $s+\mathfrak{N}$ is a non-zero-divisor of $A_{\mathrm{red}}$ where $\mathfrak{N}$ is the nil-radical of $A$. Hence, we obtain a ring map $\phi:T(A)\rightarrow T(A_{\mathrm{red}})$ which is given by $a/s\mapsto(a+\mathfrak{N})/(s+\mathfrak{N})$. It is clear that $\phi(A)=A_{\mathrm{red}}$. This shows that $\phi$ is a morphism in the category of extensions of rings. Hence, we obtain a morphism of groups $\Cl(A)\rightarrow\Cl(A_{\mathrm{red}})$. By Corollary \ref{Remark two}, the following diagram is commutative: $$\xymatrix{0\ar[r]&
\mathfrak{C}(A)\ar[r]\ar[d]&\Pic(A)\ar[d]^{\simeq}&
\\ 0\ar[r]&\mathfrak{C}(A_{\mathrm{red}})\ar[r]&
\Pic(A_{\mathrm{red}})}$$ 
where the rows are exact and the second vertical map is an isomorphism (see Theorem \ref{Thm 100 npf}). Hence,
the canonical map $\Cl(A)\rightarrow\Cl(A_{\mathrm{red}})$ is injective.
However, unlike the Picard group, the canonical injection $\Cl(A)\rightarrow\Cl(A_{\mathrm{red}})$ is not necessarily an isomorphism in general (see Example \ref{Example TDH}). 

\begin{corollary}\label{Coro ilginch} For any ring $A$, we have the following
exact sequence of Abelian groups:
$$\xymatrix{0\ar[r]&
\mathfrak{C}(A)\ar[r]&\mathfrak{C}(A_{\mathrm{red}})\ar[r]&
\mathfrak{C}(T(A)_{\mathrm{red}}, T(A_{\mathrm{red}})).}$$  
\end{corollary}

\begin{proof} The kernel of the canonical ring map $T(A)\rightarrow T(A_{\mathrm{red}})$ is the nil-radical of $T(A)$. Hence, we obtain extensions of (reduced) rings $A_{\mathrm{red}}\subseteq T(A)_{\mathrm{red}}\subseteq T(A_{\mathrm{red}})$.
Then by applying Theorem \ref{Theorem fourth TS}, we obtain the following exact sequence of Abelian groups: $$\xymatrix{0\ar[r]&
\mathfrak{C}(A_{\mathrm{red}},
T(A)_{\mathrm{red}})\ar[r]&\mathfrak{C}(A_{\mathrm{red}})
\ar[r]&\mathfrak{C}(T(A)_{\mathrm{red}}, T(A_{\mathrm{red}})).}$$
Next in Theorem \ref{theorem TH}, by taking $B=T(A)$ we obtain the canonical isomorphism of groups $\Cl(A)\simeq\Cl(A_{\mathrm{red}},
T(A)_{\mathrm{red}})$.
Thus we obtain the following exact sequence of Abelian groups: $$\xymatrix{0\ar[r]&
\mathfrak{C}(A)\ar[r]&\mathfrak{C}(A_{\mathrm{red}})\ar[r]&
\mathfrak{C}(T(A)_{\mathrm{red}}, T(A_{\mathrm{red}})).}$$
This completes the proof.
\end{proof}

\begin{example}\label{Example TDH} We give an example of a ring $A$ such that the groups $\Cl(A)$ and $\Cl(A_{\mathrm{red}})$ are not isomorphic. In particular, the canonical injection $\mathfrak{C}(A)\rightarrow\mathfrak{C}(A_{\mathrm{red}})$ is not surjective. \\
First recall that if $M$ is a module over a ring $R$, then the zero-divisors of the idealization ring $R\oplus M$ are precisely of the form $(a,x)$ with $a\in Z(R)\cup Z(M)$ and $x\in M$ where $Z(M)=\{r\in R: \exists 0\neq m\in M, rm=0\}$ is the set of zero-divisors of $M$. Indeed, if $a\in Z(R)$ then $ab=0$ for some $0\neq b\in R$. If $b\in \Ann_{R}(M)$ then $(a,x)\cdot(b,0)=0$. But if $b\notin\Ann_{R}(M)$ then $by\neq0$ for some $y\in M$. Thus $(a,x)\cdot(0,by)=0$. Then we always have  $(a,x)\in Z(R\oplus M)$. \\
So we have the canonical isomorphism of rings $T(R\oplus M)\simeq S^{-1}R\oplus S^{-1}M$ where $S=R\setminus\big(Z(R)\cup Z(M)\big)$. Also note that, as an ideal of $R\oplus M$, we have $M^{2}=0$. \\
Considering these, we are now ready to give the desired example: Let $R$ be an integral domain whose Picard group is non-trivial and consider the $R$-module $M=\bigoplus\limits_{\mathfrak{m}\in\Max(R)}R/\mathfrak{m}$
where the direct sum is taken over the set of all  maximal ideals of $R$. Now consider the idealization $A=R\oplus M$. We have $Z(R)=\{0\}$ and $Z(M)=\bigcup\limits_{\mathfrak{m}\in\Max(R)}\mathfrak{m}$.  It follows that $S=R\setminus\big(Z(R)\cup Z(M)\big)=R^{\ast}$ is the group of units of $R$. Hence, $T(A)=A$. Thus $\mathfrak{C}(A)=0$. But $A_{\mathrm{red}}=R$ and so $\mathfrak{C}(A_{\mathrm{red}})=\mathfrak{C}(R)=
\Pic(R)\neq0$. In particular, this example shows that the ideal class group in some cases behaves differently than the Picard group. 
\end{example}

The union of the minimal primes of a ring always lies in the set of its zero-divisors. In the case of equality we have the following result:   

\begin{corollary}\label{Coro ilginch 2} Let $A$ be a ring. If  $Z(A)=\bigcup\limits_{\mathfrak{p}\in\Min(A)}\mathfrak{p}$ then the canonical morphism of groups $\Cl(A)\rightarrow\Cl(A_{\mathrm{red}})$ is an isomorphism.
\end{corollary}

\begin{proof} Since $A_{\mathrm{red}}$ is a reduced ring, $Z(A_{\mathrm{red}})=\bigcup\limits_{\mathfrak{p}\in\Min(A)}
\mathfrak{p}/\mathfrak{N}$ where $\mathfrak{N}$ is the nil-radical of $A$. Thus if $b+\mathfrak{N}$ is a non-zero-divisor of $A_{\mathrm{red}}$ then $b$ is a non-zero-divisor of $A$. Hence, the canonical injection $T(A)_{\mathrm{red}}\subseteq T(A_{\mathrm{red}})$ is also surjective. But for any ring $R$ we have $\Cl(R,R)=0$ and so $\mathfrak{C}(T(A)_{\mathrm{red}}, T(A_{\mathrm{red}}))=0$. Then the assertion follows from Corollary \ref{Coro ilginch}.
\end{proof}

Regarding the above result, note that in addition to reduced rings, there are also non-reduced rings whose zero-divisors are unions of minimal primes. 

\begin{remark} For any ring $A$ we have the following commutative diagram of Abelian groups: $$\xymatrix{0\ar[r]&
\mathfrak{C}(A)\ar[r]\ar[d]&\Pic(A)\ar[r]\ar[d]^{\simeq}&
\Pic(T(A))\ar[d] \\ 0\ar[r]&\mathfrak{C}(A_{\mathrm{red}})\ar[r]&
\Pic(A_{\mathrm{red}})\ar[r]&\Pic(T(A_{\mathrm{red}}))}$$
where the rows are exact (see Corollary \ref{Coro 1 new T}) and the second vertical map is an isomorphism (Theorem \ref{Thm 100 npf}).
If the third vertical map is injective, then by the five lemma, the canonical group map $\Cl(A)\rightarrow\Cl(A_{\mathrm{red}})$ will be an isomorphism. However, as seen in Example \ref{Example TDH}, this will not happen very often.  
\end{remark}

\begin{corollary}\label{Coro dokuz} If $M$ is a module over a ring $A$, then $\Cl(A,A\oplus M)=0$ where $A\oplus M$ is the idealization of $M$.
\end{corollary} 

\begin{proof} The ring extension $A\subseteq A\oplus M$ has a retraction. In fact, the ring map $A\oplus M\rightarrow A$ given by $(r,m)\mapsto r$ is its retraction. Then apply Corollary \ref{Lemma TH 2}.
\end{proof}

\begin{example}\label{Example 3 three} Let $f:A\rightarrow B$ be a morphism of rings. We know that the kernel of the map $\Pic(f):\Pic(A)\rightarrow\Pic(B)$ is the set of all $L\in\Pic(A)$ such that $L\otimes_{A}B\simeq B$ as $B$-modules (i.e., this map must be a $B$-linear isomorphism). \\
It is important to note that if for some $L\in\Pic(A)$ we have $L\otimes_{A}B\simeq B$ as $A$-modules, then we cannot necessarily deduce that $L\in\Ker(\Pic(f))$. \\
As an example, take a ring $A$ whose Picard group is non-trivial and assume $L$ is a non-identity element of the group $\Pic(A)$. Then consider the idealization ring $B=A\oplus(\bigoplus\limits_{n\neq0}L^{\otimes n})$ where the second direct sum is taken over the set of nonzero integers. Then as $A$-modules we have $L\otimes_{A}B\simeq L\oplus
(\bigoplus\limits_{n\neq0}L^{\otimes(n+1)})
\simeq A\oplus(\bigoplus\limits_{k\neq0}L^{\otimes k})=B$. Thus $L\otimes_{A}B\simeq B$ as $A$-modules.
But $L\otimes_{A}B$ and $B$ are not isomorphic as $B$-modules. 
In other words, $L\notin\Ker(\Pic(j))$, because by Theorem \ref{Thm TRS 2} and Corollary \ref{Coro dokuz}, we have  $\Ker(\Pic(j))\simeq\mathfrak{C}(A,B)=0$, where $j:A\rightarrow B$ is the inclusion ring map. 
\end{example}

\begin{corollary} For any extension of rings $A\subseteq B$, then $\Cl(A,B)=\Cl(A, B\otimes_{A}B)$.
\end{corollary} 

\begin{proof} By applying Theorem \ref{Theorem fourth TS} to the extensions of rings $A\subseteq B\subseteq B\otimes_{A}B$ we obtain the following exact sequence of Abelian groups: $$\xymatrix{0\ar[r]&
\mathfrak{C}(A, B)\ar[r]&\mathfrak{C}(A,B\otimes_{A}B)\ar[r]&
\mathfrak{C}(B, B\otimes_{A}B).}$$
But the ring extension $B\subseteq B\otimes_{A}B$ has a retraction. Then by Corollary \ref{Lemma TH 2}, $\Cl(B, B\otimes_{A}B)=0$. 
\end{proof}

The ring defined in \cite[\S2]{Spirito} for extensions of integral domains can be generalized to arbitrary extensions of rings. More precisely, if $A\subseteq C$ and $B\subseteq C$ are extensions of rings then $AB$, the set of all finite sums of the form $\sum\limits_{k=1}^{n}a_{k}b_{k}$ with $n\geqslant1$, $a_{k}\in A$ and $b_{k}\in B$ for all $k$, is a subring of $C$. In this case, we also have the extensions of rings $A\subseteq AB$ and $B\subseteq AB$. Then the following result can be viewed as a generalization of \cite[Proposition 3.8]{Spirito}.

\begin{corollary} Let $A\subseteq B$ be an extension of rings which has a retraction and preserves non-zero-divisors. If $R$ is a subring of $T(A)$, then $\mathfrak{C}(R, BR)=0$. 
\end{corollary}

\begin{proof} To prove the assertion, by Corollary \ref{Lemma TH 2}, it suffices to show that the ring extension $R\subseteq RB$ has a retraction. First note that, since $A\setminus Z(A)\subseteq B\setminus Z(B)$, we have the extension of rings $T(A)\subseteq T(B)$. Hence, we really do have the ring $RB$. \\
Since the extension $A\subseteq B$ has a retraction, there is a ring map $f:B\rightarrow A$ such that $f(a)=a$ for all $a\in A$. \\
Then we claim that the relation $g:RB\rightarrow R$ given by $\sum\limits_{i=1}^{n}r_{i}b_{i}\mapsto\sum
\limits_{i=1}^{n}r_{i}f(b_{i})$ is well-defined (i.e., it is a function). Suppose that $\sum\limits_{i=1}^{n}r_{i}b_{i}=
\sum\limits_{k=1}^{d}r'_{k}b'_{k}$ where the $r_{i}, r'_{k}\in R$, $b_{i},b'_{k}\in B$ and $n,d\geqslant1$. Since $R \subseteq T(A)$, we may write $r_{i}=a_{i}/s_{i}$ and $r'_{k}=a'_{k}/s'_{k}$  where $a_{i}, a'_{k}\in A$ and $s_{i}, s'_{k}\in A\setminus Z(A)$. It follows that $s'\sum\limits_{i}t_{i}a_{i}b_{i}=
s\sum\limits_{k}t'_{k}a'_{k}b'_{k}\in B$ where $s=\prod\limits_{i=1}^{n}s_{i}$, $s'=\prod\limits_{k=1}^{d}s'_{k}$, 
$t_{i}=\prod\limits_{j\neq i}s_{j}$ and $t'_{k}=\prod\limits_{\ell\neq k}s'_{\ell}$. Then we get that $s'\sum\limits_{i}t_{i}a_{i}f(b_{i})=
s\sum\limits_{k}t'_{k}a'_{k}f(b'_{k})$. This yields that 
$ss'\big(\sum\limits_{i}r_{i}f(b_{i})-
\sum\limits_{k}r'_{k}f(b'_{k})\big)=0$. But $ss'$ is a non-zero-divisor of $A$ and $\sum\limits_{i}r_{i}f(b_{i})-
\sum\limits_{k}r'_{k}f(b'_{k})\in T(A)$. Then
$\sum\limits_{i}r_{i}f(b_{i})=
\sum\limits_{k}r'_{k}f(b'_{k})$. This establishes the claim. This shows that $g$ is a function. Then it can be easily seen that $g$ is a morphism of rings and $g(r)=r$ for all $r\in R$.     
\end{proof}

\begin{remark}\label{mini remark} Assume an extension of rings $A\subseteq B$ has a retraction. Then there is a morphism of rings $f:B\rightarrow A$ such that $f(a)=a$ for all $a\in A$. If $I$ is an ideal of $B$ such that $I\subseteq\Ker(f)$ then it is clear that the extension of rings $A\subseteq B/I$ also has a retraction.
\end{remark}

Let $A$ be a ring such that the extension $T(A)_{\mathrm{red}}\subseteq T(A_{\mathrm{red}})$ has a retraction, then using Corollaries \ref{Lemma TH 2} and \ref{Coro ilginch}, we have the canonical isomorphism of groups $\mathfrak{C}(A)\simeq\mathfrak{C}(A_{\mathrm{red}})$. This can be viewed as a generalization of Corollary \ref{Coro ilginch 2}.

\begin{remark}\label{Remark three} Let $A\subseteq B\subseteq C$ and $A'\subseteq B'\subseteq C'$ be extensions of rings and let $\phi:C\rightarrow C'$ be a ring map such that $\phi(A)\subseteq A'$ and $\phi(B)\subseteq B'$. Then in the category of extensions of rings, we have the following commutative diagram:
$$\xymatrix{(A,C)\ar[r]\ar[d]^{\phi}&(B,C)\ar[d]^{\phi} \\ (A',C')\ar[r]&(B', C').}$$
Then the functoriality of the group of invertible ideals, gives us the following commutative diagram of Abelian groups: $$\xymatrix{
\mathscr{G}(A,C)\ar[r]\ar[d]&\mathscr{G}(B,C)\ar[d] \\ \mathscr{G}(A',C')\ar[r]&
\mathscr{G}(B', C').}$$
\end{remark}

The sequence in Theorem \ref{Theorem fourth TS} is not right exact in general. In other words, the canonical map $\Cl(A,C)\rightarrow\Cl(B,C)$ is not necessarily surjective. First note that for any extensions of rings $A\subseteq B\subseteq C$ the following diagram is commutative: $$\xymatrix{
\mathfrak{C}(A,C)\ar[r]\ar[d]^{} &\Pic(A)\ar[d]^{} \\ \mathfrak{C}(B,C)\ar[r]^{}&\Pic(B)\ar[r]&\Pic(C).}$$ 
Indeed, if $I\in\mathscr{G}(A,C)$ then by Lemma \ref{Lemma 5 besh}, $I\otimes_{A}B\simeq IB$ as $B$-modules.
In the above diagram, the bottom row is also exact (see Theorem \ref{Thm TRS 2}). Now for instance, if the Picard groups of $A$ and $C$ are trivial but $\Pic(B)\neq0$ (see Example \ref{Example one}), then the map $\Cl(A,C)\rightarrow\Cl(B,C)$ is not surjective. 

However, we will observe that the canonical map $\Cl(A,C)\rightarrow\Cl(B,C)$ is surjective in some cases. To this end, let us first recall from \cite[\S2]{Swan 2} that a ring extension $A\subseteq B$ is called subintegral if it is integral, the induced map between the corresponding prime spectra is bijective and the extensions between the residue fields are isomorphisms. We have then the following result: 

\begin{corollary}\label{Coro Ani-Elman-Asena} Let $A\subseteq B\subseteq C$ be extensions of rings such that one of the following conditions hold: \\
$\mathbf{(i)}$ the canonical map $\Cl(A_{\mathrm{red}},C_{\mathrm{red}})
\rightarrow\Cl(B_{\mathrm{red}},C_{\mathrm{red}})$ is surjective. \\
$\mathbf{(ii)}$ the extension $A\subseteq C$ is subintegral. \\
Then we have the following short exact sequence of Abelian groups:  $$\xymatrix{0\ar[r]&
\mathfrak{C}(A,B)\ar[r]&\mathfrak{C}(A,C)\ar[r]&
\mathfrak{C}(B,C)\ar[r]&0.}$$
\end{corollary} 

\begin{proof} First, suppose that (i) holds. To prove the assertion, by Theorem \ref{Theorem fourth TS}, it suffices to show that the canonical map $\Cl(A,C)\rightarrow\Cl(B,C)$ is surjective. By Remark \ref{Remark three}, the following diagram is commutative: $$\xymatrix{
\mathfrak{C}(A,C)\ar[r]\ar[d]&\mathfrak{C}(B,C)\ar[d] \\ \mathfrak{C}(A_{\mathrm{red}},C_{\mathrm{red}})\ar[r]&
\mathfrak{C}(B_{\mathrm{red}}, C_{\mathrm{red}})}$$
where the vertical arrows are isomorphisms (see Theorem \ref{theorem TH}). Thus the canonical map $\Cl(A,C)\rightarrow\Cl(B,C)$ is surjective. \\  
Now assume that (ii) holds. Then by \cite[Theorem 3.3]{Singh}, the canonical map $\mathscr{G}(A,C)\rightarrow\mathscr{G}(B,C)$ is surjective. Hence, the induced map
$\Cl(A,C)\rightarrow\Cl(B,C)$ is surjective. Then apply Theorem \ref{Theorem fourth TS}.
\end{proof}

In fact, for any  extensions of rings $A\subseteq B\subseteq C$, the canonical map $\Cl(A,C)\rightarrow\Cl(B,C)$ is surjective if and only if the canonical map $\Cl(A_{\mathrm{red}},C_{\mathrm{red}})
\rightarrow\Cl(B_{\mathrm{red}},C_{\mathrm{red}})$ is surjective. \\

\textbf{Acknowledgments.} We would like to give sincere thanks to Professors Brian Conrad, Pierre Deligne, and Melvin Hochster who generously shared with us their very valuable and excellent ideas.

\end{document}